\documentclass[11pt,a4paper,reqno]{amsart}
\usepackage{amsfonts}
\usepackage{amsthm}
\usepackage{amssymb} 
\usepackage{faktor} 
\usepackage{graphicx}
\usepackage{float} 
\usepackage{subcaption} 
\usepackage{amscd}
\usepackage[utf8]{inputenc}
\usepackage{t1enc}
\usepackage{dsfont}

\usepackage{xfrac} 
\usepackage{xcolor}
\usepackage{amsbsy} 
\newcommand*{\transp}[2][-3mu]{\ensuremath{\mskip1mu\prescript{\smash{\mathrm t\mkern#1}}{}{\mathstrut#2}}}

\newcommand{\defeq}{\mathrel{\mathop:}=}

\DeclareRobustCommand{\stirling}{\genfrac\{\}{0pt}{}}
\usepackage[all]{xy} 
\usepackage{mathtools} 
\usepackage[mathscr]{eucal}
\usepackage{indentfirst}
\usepackage{graphicx}
\usepackage{pict2e}
\usepackage{epic}
\numberwithin{equation}{section}
\usepackage[margin=2.9cm]{geometry}
\usepackage{epstopdf} 
\usepackage[bottom]{footmisc} 
\usepackage{array}

\newcommand{\Wn}{{W^{(1)}}}
\newcommand{\Wd}{W^{(2)}}

\newcommand{\gl}{\lambda}
\newcommand{\N}{\mathbb{N}}

\newcommand{\eps}{\varepsilon}

\newcommand{\cP}{\mathcal{P}}

\renewcommand{\P}{\mathbb{P}}
\newcommand{\E}{\mathbb{E}}
\newcommand{\R}{\mathbb{R}}

\theoremstyle{plain}
\newtheorem{Th}{Theorem}[section]
\newtheorem{Lemma}[Th]{Lemma}

\newtheorem{Prop}[Th]{Proposition}
\newtheorem{rem}[Th]{Remark}
\theoremstyle{definition}

\newtheorem{Rem}[Th]{Remark}

\usepackage{hyperref}

\begin{document}
	\title{Long-Range Correlation of the Sine$_\beta$ point Process}
	\author{Laure Dumaz}
	\address{CNRS \& Department of Mathematics and Applications, \'Ecole Normale Sup\'erieure (Paris), 45 rue d’Ulm, 75005 Paris, France}
	\email{laure.dumaz@ens.fr}
	\author{Martin Malvy}
	\address{Ceremade, University Paris Dauphine, Place du Mar'echal de Lattre de Tassigny, 75016 Paris \& Department of Mathematics and Applications, \'Ecole Normale Sup\'erieure (Paris), 45 rue d’Ulm, 75005 Paris, France}
	\email{martin.malvy@ens.fr}

	\date{\today}

	\begin{abstract} We study the correlations of the celebrated Sine$_\beta$ point process. This point process arises as the bulk scaling limit of $\beta$-ensembles and has a geometric description through the Brownian carousel, as shown by Valk\'o and Vir\'ag (2009). 
		
		We establish that the averaged $k$-point truncated correlation functions decay polynomially in the limit of large separation. We show that the decay exponent is of order $1/\beta$ for large $\beta$. This is a step towards a conjecture by Forrester and Haldane regarding the exact asymptotics of the two-point correlation function, a problem  recently addressed by Qu and Valk\'o (2025). Our proofs, which rely on a careful analysis of the coupling of diffusions associated with the Brownian carousel, hold for all $\beta >0$ and $k \geq 1$, significantly extending previous results limited to specific values of $\beta$ or $k$. 
	\end{abstract}
	
	\maketitle
	\tableofcontents
	
	\section{Introduction}
	\subsection{Log-gases and random matrices} 	
	
	The \text{Sine}$_\beta$ point process is a fundamental object arising in Random Matrix Theory (RMT) and statistical mechanics. In RMT, it describes the bulk scaling limit of the eigenvalue distributions of the celebrated invariant Gaussian ensembles for $\beta = 1,2,4$, as well as the general $\beta >0$ tridiagonal ensembles introduced by Dumitriu and Edelman \cite{dumitriu_matrix_2002}.
	Their joint probability distribution is given by the Gibbs measure
	\begin{align}\label{beta ensemble}
		\frac{1}{Z_{n,\beta}}\prod_{1\leq i<j \leq n}|\lambda_i-\lambda_j|^\beta \prod_{i=1}^n e^{-\beta \lambda_i^2/4}\,\mathrm{d}\gl_i\,.
	\end{align}
	
	In statistical mechanics, the distribution \eqref{beta ensemble} represents a one-dimensional gas of particles interacting via a two-dimensional Coulomb (logarithmic) repulsion, confined by a quadratic potential, at inverse temperature $\beta$. These systems, known as \textit{log-gases}, belong to the broader class of Riesz gases (see \cite{lewin_coulomb_2022} for a comprehensive overview).
	Equivalently, the Sine$_\beta$ process describes a gas of quantum particles at zero temperature evolving on the line  and interacting via the potential $\beta(\beta -2)/|x|^2$, with the cases of $\beta = 0$ and  $\beta = 2$ respectively corresponding to free bosons and fermions. The integrability of this model was studied by Calogero, Sutherland and Moser \cite{calogero1971, sutherland1971, moser1975}.

	\subsubsection*{The Stochastic Approach: The Brownian Carousel}
	For the classical values $\beta=1,2,4$, the microscopic fluctuations of \eqref{beta ensemble} were understood early on, due to the integrable structure of the models (determinantal for $\beta=2$ and Pfaffian for $\beta=1,4$, see e.g., \cite{mehta_random_2004}). 
	
	Significant progress was made by Valk\'o and Vir\'ag \cite{valko_continuum_2009} in understanding the case of arbitrary $\beta >0$. Using the tridiagonal matrix models of \cite{dumitriu_matrix_2002}, they proved the convergence of the rescaled point process in the bulk of the spectrum. Specifically, for any energy level $E \in (-2,2)$, the recaled point process
	$$
	\sum_{k = 1}^n \delta_{n\sqrt{4-E^2}\;(\lambda_k/\sqrt{n}-E)}
	$$
	converges vaguely as $n \to \infty$ towards a non-trivial, translation-invariant random point process: the \text{Sine}$_\beta$ point process
	(see also \cite{killip_eigenvalue_2006} for an alternative construction using Circular $\beta$-ensembles).
	
	Moreover their approach gives a geometrical interpretation of the Sine$_\beta$ point process via the Brownian carousel. The counting function of the \text{Sine}$_\beta$ process is described by $\lambda \mapsto \alpha_\lambda(+\infty)/2\pi$, where the phases $(\alpha_\lambda)_{\lambda\in\mathbb{R}}$ are solutions to a family of coupled stochastic differential equations driven by a complex Brownian motion. 	This construction directly implies that $\text{Sine}_\beta$ is invariant under translations, with intensity $1/(2\pi)$ but it also gives access to fine statistics, including large gap probabilities \cite{valko_large_2010} and Central Limit Theorems \cite{kritchevski_scaling_2012}.

	\subsubsection*{The Variational Approach: Renormalized Energy}
	
	To understand the microscopic behavior of log-gases, Serfaty and Leblé \cite{leble_large_2017} developed a variational framework based on the ``Renormalized Energy'' of point processes $\mathcal{P}$, denoted by $\mathbb{W}(\mathcal{P})$. They proved that, up to extraction, the averaged microscopic point process converges to point processes $\mathcal{P}$ that minimize the free energy functional:
	\begin{equation}\label{free energy}
		\mathcal{F}_\beta(\mathcal{P}) = \beta \mathbb{W}(\mathcal{P}) + \mathrm{ent}[\mathcal{P}|\Pi],
	\end{equation}
	where $\mathrm{ent}[\cdot|\Pi]$ denotes the specific relative entropy with respect to a Poisson point process $\Pi$ of intensity $1/(2\pi)$. 

	The \text{Sine}$_\beta$ point process is then identified as the unique minimizer of $\mathcal{F}_\beta$ among translation-invariant processes \cite{erbar2018}.

	\subsubsection*{The Statistical Physics Perspective: DLR Equations}

	In \cite{dereudre_dlr_2021}, Dereudre, Hardy, Leblé and Maïda obtain an alternative description of Sine$_\beta$ as an infinite Gibbs measure at inverse temperature $\beta >0$ associated with the logarithmic pair potential interaction. More precisely, they show that \text{Sine}$_\beta$ solves the canonical Dobrushin-Lanford-Ruelle (DLR) equations for a specific renormalized logarithmic potential.

	While it is known that \text{Sine}$_\beta$ solves these DLR equations, the uniqueness of the solution remains an open question for general $\beta$, except for the case $\beta=2$ due to the determinantal structure.
	
	\subsection{Main results of the paper}
	
	Consider a point process such that
	\begin{align*}
\E\Big[ (\# \mbox{ points in }B)^k \Big]  < \infty\,.
	\end{align*}
	for all $k$ and bounded Borel set $B$.
	Then, its $k$-point correlation measure $\rho^{(k)}$
	is defined such that for any \emph{pairwise disjoint} Borel sets $B_1,\dots, B_k$, one has:
	\begin{align*}
		\E\Big[ \prod_{i=1}^k (\# \mbox{ points in }B_i) \Big] = \rho^{(k)}(B_1 \times \dots \times B_k)\,.
	\end{align*}  
	(See paragraph \ref{subsec:correlations} below for a more precise definition). We will use the notation $\text{Sine}_\beta(B)$ for the number of points of a point process with law $\text{Sine}_\beta$ in a given Borel set $B$. In particular, $\E[\text{Sine}_\beta(B)] = \rho^{(1)}(B)$.
	
	A consequence of the DLR equations is that for all $\beta>0$, such correlation measures have a density with respect to Lebegue measure, called correlation functions, see \cite{Malvy_these} for more details. For the \text{Sine}$_\beta$ point process, with a slight abuse of notation, we simply denote both the measure and the density by $\rho^{(k)}$ in the rest of the article.
	
	\medskip
	
	The main objective of the present article is to prove asymptotic bounds on the correlation functions for the Sine$_\beta$ point process when the points are widely separated, and which hold for all parameters $\beta >0$.

	\subsubsection*{Two-point correlation function}   
	
	Let us first state our result for $k=2$ as it is the most insightful case.   
	For any $\beta>0$, we define the \textit{truncated} two-point correlation function as
	$$\rho^{(2)}_T(x,y)=\rho^{(2)}(x,y)-\rho^{(1)}(x)\rho^{(1)}(y)\,.$$  In our case, note that $\rho^{(1)}(x) = 1/(2 \pi)$. Since $\text{Sine}_\beta$ is translation invariant, we will slightly abuse notation in the following by writing $\rho^{(2)}_T(r)$ for $\rho^{(2)}_T(x, x+r)$. We analyze in this paper a locally averaged version of the truncated correlation function and show its polynomial decay.
	\begin{Th}[Decay of the truncated two-point correlation function]\label{theo:main}
		There exist an absolute constant $c>0$ and for all $\beta, \lambda_0>0$ a constant $C=C(\beta,\lambda_0)>0$ so that for any $\lambda\leq \lambda_0$, $r\geq 1$,
		\begin{equation*}\Big|\int_0^{\gl} \rho^{(2)}_T(x+r) \mathrm{d}x\Big|\leq C r^{-\frac{c\beta}{1+\beta^2}}\,.
		\end{equation*}
	\end{Th}

	This result relates to a conjecture by Forrester and Haldane in \cite{forrester2010}. Using the Calogero-Sutherland Hamiltonian (see e.g. Chapter 5 of the book of Sutherland \cite{sutherland2004}) and Haldane's theory about compressible quantum fluids \cite{haldane_effective_1981}, Forrester predicted the asymptotic behavior for the two-points truncated correlation function. He established this rigorously for even integers values of $\beta$ (see Proposition 13.2.4 in \cite{forrester2010}) and derived analogous results for rational $\beta$ under a truncation scheme  (see Formula 13.226 in \cite{forrester2010}). Extracting the leading-order behavior, the asymptotics is predicted to be
\begin{equation}\label{asympforrester}
	\rho^{(2)}_T(r) \sim
	\begin{cases}
		-\dfrac{1}{\pi^2 \beta r^2} & \text{for } \beta < 2\, , \\[15pt]
		-\dfrac{1}{2 \pi^2 r^2} + \dfrac{\cos r}{2 \pi^2  r^2} & \text{for } \beta = 2\, , \\[15pt]
		a_1 \dfrac{\cos r}{r^{4/\beta}} & \text{for } \beta > 2 \,,
	\end{cases}
\end{equation}
with an explicitly given coefficient $a_1$.

	\medskip

	Our Theorem \ref{theo:main} specifies that the decay rate decreases for large $\beta$ as $c/\beta$. This behavior is not surprising as the Sine$_\beta$ point process crystallizes towards the ``Picket Fence'' configuration $U + 2 \pi \mathbb{Z}$ (where $U$ is uniform on $[0,2\pi]$) as $\beta \to \infty$. Conversely, as $\beta$ tends to $0$, $\text{Sine}_\beta$ converges to a Poisson point process (see \cite{allez_sine_2014} and \cite{erbar2018}), for which the truncated correlation functions vanish. Here from the Forrester/Haldane conjecture, one expects that the higher coefficient is equal to $2$ but we obtain only some exponent of order $O(\beta)$.
	
Note that when $\beta = 2$, the $\text{Sine}_\beta$ process has a determinantal structure and
	$$\rho^{(k)}(x_1,\cdots,x_k)=\det\left(K(x_i,x_j)\right)_{1\leq i,j\leq k}, \qquad K(x,y)\defeq\frac{\sin((x-y)/2)}{\pi|x-y|}. $$
	In particular, the two-points truncated function $\rho^{(2)}_T$ admits the explicit formula
	$$\rho^{(2)}_T(r)=-\frac{\sin^2(r/2)}{\pi^2 r^2}.$$
	A  \textit{Berezinskii--Kosterlitz--Thouless} (BKT) phase transition is expected at $\beta=2$ (see \cite[Remark 41]{lewin_coulomb_2022})
	and this critical value is conjectured to be at the junction in between the universal decay $|x-y|^{-2}$ and oscillatory $\beta$-dependent decay, as seen in the Forrester--Haldane conjecture \eqref{asympforrester}.
	
	\medskip
	
		Very recently, a stronger version of Theorem \ref{theo:main} was obtained by Qu and Valk\'o in a preprint \cite{qu2025} that appeared on ArXiv while we were finishing the present paper. They derived the exact first order asymptotics (equal to $r^{-4/\beta}$) for $\beta \geq 4$, and obtained non-optimal bounds for small values of $\beta$, similar to our own.  However, our methodology differs significantly from theirs. Their approach relies on an exact correspondence between the two points function of Sine$_\beta$ and the one-point function of the Hua Pickrell process with parameters $(\beta, \beta/2)$ (which describes Sine$_\beta$ conditioned to have a point at the origin). In contrast, our strategy involves a probabilistic alternative to these algebraic identities. This will allow us to treat the higher $k$-point correlation functions for $k \geq 2$.

	\medskip
	
	Even more recently, Assiotis and Najnudel \cite{AssiotisNajnudel2026} established an exact identity of the $k$-point correlation function in terms of some moments of the Hua-Pickrell stochastic zeta function (defined by Li and Valk\'o in \cite{LiValko2022}) with parameters $(\beta, k\beta/2)$. While exact, extracting the large separation asymptotics from this identity seems to remain highly non-trivial.

	\subsubsection*{General $k$-point correlation function}  
	Our main result generalizes Theorem \ref{theo:main} to arbitrary $k$-point correlation functions. Let us fix $k \geq 2$ and $1 \leq k_0 < k$. We consider the \textit{partially truncated functions} associated with two given clusters
	$$\rho^{(k)}(x_1,\cdots,x_k)-\rho^{(k_0)}(x_1,\cdots,x_{k_0})\rho^{(k-k_0)}(x_{k_0+1},\cdots, x_k)\, , $$
	for $x_1,\cdots,x_k \in \R$.
	Heuristically, this quantity measures the ``correlation'' between the first cluster $(x_1,\cdots,x_{k_0})$ and the second cluster $(x_{k_0+1},\cdots, x_k)$. It would vanish for a Poisson point process. For any scalar $r\in\mathbb{R}$ and $\mathbf{x} = (x_1,\cdots,x_d)\in\mathbb{R}^d$, we use the notation $\mathbf{x}+r$ for the translated vector $(x_1+r,\cdots,x_d+r)$. In the following result, we keep track of how the pre-factor to the decay rate depends on the parameter $k$ and the cluster size.
	\begin{Th}[Decay of partially truncated $k$-point correlations]\label{theo:main2}
		There exist an absolute positive constant $c>0$ and, for any $\beta >0$ and $\lambda_0 >0$, a constant $C=C(\beta,\lambda_0)>0$ such that the following holds.
		Fix $k \geq 2$ and $1 \leq k_0 < k$. 
		Let $\mathcal{I}_1 \subset (\mathbb{R}^-)^{k_0}$ and $\mathcal{I}_2 \subset (\mathbb{R}^+)^{k-k_0}$ be two products of disjoint intervals where each interval has length at most $\lambda_0$. Let $L$ be the diameter of the union of the intervals composing $\mathcal{I}_1$ and $\mathcal{I}_2$. Then, we have 	for all $r \geq 1$,
		\begin{align}\Big|\int_{\mathcal{I}_1}\int_{\mathcal{I}_2} \big[\rho^{(k)}(\mathbf{x},\mathbf{y}+r)-\rho^{(k_0)}(\mathbf{x})\rho^{(k-k_0)}(\mathbf{y}+r)\big]\;\mathrm{d}\mathbf{x}\,\mathrm{d}\mathbf{y}\Big|
			\leq \, K(k,L)
			\; r^{-c\frac{\beta}{1+\beta^2}}, \label{twoclusterbound}
		\end{align}
		where the pre-factor $K(k,L)$ depends on $k, L$ and the constant $C$ as follows:
		\begin{align*}
			K(k,L)= \min(C k^{3} e^{C L}, C^k k^{k/2} )\,.
		\end{align*}
	\end{Th}
	\vspace{0,2cm}

	\begin{Rem}     	
		The clusters $\mathcal{I}_1 : = \prod_{i=1}^{k_0} I_i$ and $\mathcal{I}_2 : = \prod_{i=k_0 + 1}^{k} I_i$ are chosen in $\R_-$ (resp. $\R_+$) so that the distance between $\cup_{i \leq k_0} I_i$ composing $\mathcal{I}_1$ and $\cup_{i=k_0 +1}^k (I_i + r)$ composing the shifted by $r$ version of $\mathcal{I}_2$ is at least $r$.
		Due to the translation invariance of the correlation functions, the same result holds for any shifted domains $\mathcal{I}_1 + t : = \prod_{i=1}^{k_0} (I_i + t)$ and $\mathcal{I}_2 + t : = \prod_{i=k_0 + 1}^{k} (I_i + t)$ with $t\in\mathbb{R}$.
		The most common pair-clustering scenario corresponds to sending a single particle far away that is $k-k_0=1$. Recalling that $\rho^{(1)} \equiv (2\pi)^{-1}$, the bound \eqref{twoclusterbound} becomes
		\begin{equation}\label{eq: partially truncated}
			\left|\int_{\mathcal{I}_1}\int_{0}^{\lambda} \left( \rho^{(k)}(\mathbf{x}, y+r) - \frac{\rho^{(k-1)}(\mathbf{x})}{2\pi} \right) \;\mathrm{d}\mathbf{x}\,\mathrm{d}y\right| \leq K(k, L)\; r^{-c\frac{\beta}{1+\beta^2}}.
		\end{equation}
	\end{Rem}

	\medskip
	
	\begin{Rem}     	
		Let us analyze the behavior of the pre-factor $K(k, L)$ in different regimes of $k$ and $L$. On one hand, if $k$ is bounded, $K(k, L)$ is uniformly bounded in $L$. On the other hand, for $k$ large, the  situation goes as follows.
		\begin{enumerate}
			\item \textit{Small cluster size: $L$ is fixed or $L = O(\ln k)$}.
			In this case, our bound grows polynomially in $k$. Note that fitting $k$ disjoint intervals into a domain of size $O(\ln k)$ implies that the average interval length is $O(\frac{\ln k}{k})$. This corresponds to a rare ``overcrowding'' event.
			\item \textit{Intermediate cluster size: $\ln k \ll L \ll k \ln k$}.
			In this regime, our bound grows exponentially with $L$. The case where $L$ is of order $k$ represents the typical scenario for the point process (where each point lies in an interval of size $O(1)$).
			\item \textit{Arbitrarily large cluster size: $L \gg k \ln k$.}
			Here, the bound is dominated by the second term, yielding a sub-factorial growth of $k^{k/2}$.
			It is instructive to compare this growth with that of the moments of the number of points inside the intervals. By Hölder's inequality, for any disjoint intervals $I_i$ of length bounded by $\gl_0$,
			$$
			\int_{I_1\times\cdots\times I_k} \rho^{(k)}(\mathbf{x})\mathrm{d}\mathbf{x} = \mathbb{E}\Big[\prod_{i=1}^{k}\mbox{Sine}_\beta(I_i)\Big] \leq \mathbb{E}\left[(\mbox{Sine}_\beta([0,\lambda_0]))^k\right].
			$$
			As shown in Proposition \ref{overcrowdingTh}, the random variable $\mbox{Sine}_\beta([0,\lambda_0])$ is sub-Gaussian, hence its moments satisfy $\mathbb{E}[(\mbox{Sine}_\beta([0,\lambda_0])^k] \leq C^k k^{k/2}$.
			Since the uniform bounds $K(k,\infty)$ satisfy the same growth rate, they can be viewed as optimal in this sense.
		\end{enumerate}
	\end{Rem}
	\vspace{0,2cm}

	The previous results directly imply similar bounds for the $k$-point truncated correlation functions. These are defined recursively as follows: $\rho_T^{(1)}(x_1) = \rho^{(1)}(x_1)$ and
	\begin{align}\label{cumulants equivalent}
		\rho_T^{(k)}(x_1,\cdots,x_k)=\rho^{(k)}(x_1,\cdots,x_k)-\sum_{\pi = (J_1, \cdots, J_j),\ \pi \neq \{1,\cdots,k\}}\rho^{(|J_1|)}_T(X_{J_1})\cdots\rho_T^{(|J_j|)}(X_{J_j})\,,
	\end{align}
	where the sum runs over all the partitions $\pi$ of $\{1,\cdots,k\}$ except the trivial one $\{1,\cdots,k\}$, and $X_{(i_1,\cdots,i_j)}:=(x_{i_1},\cdots,x_{i_j})$. In the context of point processes, correlation functions act as the natural analogs of moments of a random variable, while truncated correlation functions correspond to their cumulants. Indeed, the $k$-point truncated correlation function describes only the $k$-body correlations, by factoring out the correlations from smaller subsets.

	From Theorem \ref{theo:main2}, one can obtain as well the decay of the $k$-point truncated correlation functions, upon changing the pre-factor $K(k,L)$:
	\begin{Th}[Truncated $k$-point correlation decay]\label{cor:truncated} 
		There exist an absolute positive constant $c>0$ and for any $\beta >0$ and $\lambda_0 >0$, a constant $C=C(\beta,\lambda_0)>0$ such that the following holds.
		Let us fix $k \geq 2$ and take $(I_i)_{1 \leq i\leq k}$ pairwise disjoint intervals of size atmost $\lambda_0$. We have:
		\begin{equation*}\label{cor:truncated}
			\Big|\int_{I_1\times\cdots\times I_k}\rho^{(k)}_T(x_1,\cdots,x_k)\,\mathrm{d}x_1\cdots\mathrm{d}x_k\,\Big|\leq  K_2(k)\Big(\max_{S_1 \sqcup S_2 = \{1,\cdots, k\}}\mathrm{dist}\big(\underset{i \in S_1}{\cup}  I_i, \underset{i \in S_2}{\cup} I_i\big) \Big)^{-c\frac{ \beta}{1+\beta^2}},
		\end{equation*}
		where  $K_2(k) = C^k \, k! \, k^{k/2}$.
	\end{Th}
	This theorem is a direct consequence of Theorem \ref{theo:main2}. We refer to Section \ref{subsec:correlations} for a proof.
	
	\medskip
	
	Let us briefly discuss the implications of this last result. Informally, it establishes the asymptotic independence of regions which are spatially far. 
	For general point processes, following \cite{lanford1969observables}, the set of solutions to the DLR equations can be represented as convex combinations of extremal Gibbs states. Consequently, showing uniqueness is equivalent to proving that there is a unique extremal state. Following the framework established by Georgii \cite{georgii1976}, these extremal states are characterized by a trivial tail $\sigma$-algebra. When the interactions are short-range, this property is equivalent to the decay of partially truncated correlations. Dereudre et al. \cite{dereudre_dlr_2021} proved that the Sine$_\beta$ process satisfies the DLR equations, but uniqueness among translation-invariant solutions remains open, except in the special case $\beta=2$. Our result on the decay of truncated correlations therefore strongly suggests that Sine$_\beta$ is an extremal state for the DLR equations, and that it should be their unique solution. This strongly relates to the uniqueness result of Erbar, Huesmann, and Leblé \cite{erbar2018}, who proved that Sine$_\beta$ is the unique minimizer of the infinite-volume free energy among translation-invariant point processes.

	\subsection*{Acknowledgments}
	We would like to thank Benedek Valk\'o and Yahui Qu for useful discussions. Special thanks are due to Mathieu Lewin for his careful reading of a previous version of this article and for his numerous helpful suggestions during its preparation. L.D. acknowledges the support of ANR RANDOP ANR-24-CE40-3377 and LOCAL ANR-22-CE40-0012.

	\section{Strategy of proof and proof of the theorems}
	\subsection{The \textit{stochastic sine equations}}\label{subsec:carousel} In \cite{valko_continuum_2009}, Valk\'o and Vir\'ag give a description of the point process $\text{Sine}_\beta$ in terms of a family $(\alpha_\lambda)_{\lambda\in\mathbb{R}}$ of one-dimensional diffusion processes called the \textit{stochastic sine equations}, which describes the evolution of the hyperbolic angle determined by the Brownian carousel of parameter $\lambda$, its driving Brownian motion and its starting point, see Section 2.1 of \cite{valko_continuum_2009}. In particular, the family $(\alpha_\gl)_\gl$ satisfies the \emph{stochastic sine equation}:
	\begin{equation}\label{sine eq}
		\mathrm{d}\alpha_\lambda(t)=\lambda\frac{\beta}{4}e^{-\frac{\beta}{4}t}\mathrm{d}t+\text{Re}\left((e^{-i\alpha_\lambda(t)}-1)\mathrm{d}W(t) \right)\,,
	\end{equation}
	with initial condition $\alpha_\lambda(0)=0$ and where $(W(t))_{t\geq 0}$ is a standard complex Brownian motion.  In the latter, we will denote $f(t):=(\beta/4) e^{-\beta t/4}$.
	
	Note that the family $(\alpha_\gl(s),\; s\geq 0,\;\gl \in \R)$ has a unique solution defined in $[0,+\infty)$ such that for all $s \geq 0$, $\gl \mapsto \alpha_\gl(s)$ is continuous.
	
	An important property shared by the processes $\alpha_\lambda$ for $\lambda\in\mathbb{R}$ is that, almost surely, they all converge to multiples of $2\pi$ in the large time limit, see Figure \ref{fig:carousel} for a plot of a simulation of their trajectories for two values of $\gl$. By definition, almost surely, $\gl \mapsto \alpha_\gl(+\infty)$ is right-continuous with left limits. It corresponds to the counting measure of the Sine$_\beta$ point process:
	\begin{equation}
		\big(\,\text{Sine}_\beta\big((0,\lambda]\big) \,\big)_{\lambda \in \R}\overset{(d)}{=}\left(\frac{\alpha_{\lambda}(+\infty)}{2\pi}\right)_{\lambda \in \R}\,.
	\end{equation} 
	
	\begin{Rem}[$\beta = 2$ transition]
		The BKT transition at $\beta=2$ can be hinted in the change in the long-time behavior of the Brownian carousel according to the value of $\beta$:
		Valk\'o and Vir\'ag prove in \cite{valko_continuum_2009} that the diffusion $\alpha_\lambda$ converges from above a.s. if and only if $\beta\leq 2$. 
	\end{Rem}

	Let us recall important properties of the diffusion processes $(\alpha_\lambda)_{\lambda\in\mathbb{R}}$, already established in \cite{valko_continuum_2009}, that will be useful in this paper:
	\begin{enumerate}
		\item For all $\gl, \gl' \in \R$, $\alpha_{\lambda'}-\alpha_{\lambda}$ has the same distribution as $\alpha_{\lambda'-\lambda}$ ;\\
		
		\item For all $t \geq 0$, $\gl \mapsto \alpha_\lambda(t)$ is increasing ;\\
		
		\item For all $\gl \geq 0$, $t \mapsto \lfloor \alpha_\lambda(t) \rfloor_{2\pi}$ is non-decreasing. Here, $\lfloor x\rfloor_{2\pi}=\max(2\pi\mathbb{Z}\cap (-\infty,x])$ ;\\
		
		\item For all $\gl > 0$ and all integers $a,k$, $\displaystyle\mathbb{P}\left(\text{Sine}_\beta[0,\lambda] \geq a\,k\right)\leq 2\left(\frac{\lambda}{2\pi a}\right)^k$.
	\end{enumerate}
	
	Property $(1)$ ensures that the point process $\text{Sine}_\beta$ is invariant under translation. This is easily obtained as the process $\tilde{\alpha}\defeq\alpha_{\lambda'}-\alpha_\lambda$ is the solution of 
	\begin{align*}
		\mathrm{d}\tilde{\alpha}(t)=(\lambda'-\gl) f(t)\mathrm{d}t+\text{Re}\left((e^{-i\tilde{\alpha}(t)}-1)\mathrm{d} \tilde W(t) \right), \qquad \tilde{\alpha}(0)=0\,,
	\end{align*}
	with driving Brownian motion $\tilde W(t) =\int_0^t e^{-i\alpha_{\lambda}(s)} \mathrm{d}W(s)$. Note that we will use this representation for the difference later on.
	
	Property $(4)$ provides a control on the number of points in intervals. In particular, it implies that the $k$ factorial moments satisfy
	\begin{align*}
		\E\big[\text{Sine}_\beta[0,\eps] \times \cdots \times (\text{Sine}_\beta[0,\eps] -k+1)\big] \leq C_k \eps^k\,,
	\end{align*}
	which implies the existence of the $k$-point correlation measures.
	
	In \cite{holcomb_overcrowding_2015}, Holcomb and Valk\'o established precise tail estimates for overcrowding. The following theorem extends their bound to arbitrarily large intervals, providing a uniform estimate that we will use to obtain the pre-factor $K(k,L)$ in Theorem \ref{theo:main2}.
	\begin{Th}[Overcrowding estimate, Holcomb, Valk\'o, adapted from \cite{holcomb_overcrowding_2015}]\label{overcrowdingTh}
		For all $\beta >0$, there exists a constant $C=C(\beta)$ such that for all $\lambda \geq 1$ and all $n\geq C \lambda$,
		\begin{equation}\label{overcrowding}\mathbb{P}\big(\mathrm{Sine}_\beta[0,\gl]\geq n\big)\leq \exp\big(-(\beta/4) n^2\ln(n/\lambda)\big).
		\end{equation}
		
	\end{Th}
\begin{proof}[Proof sketch]
The authors show in \cite{holcomb_overcrowding_2015} that there exists a positive constant $c=c(\beta,\gl_0)$ such that for all $0<\lambda \leq \gl_0$ and for $n\geq 1$, one has
	$$\mathbb{P}\big(\mathrm{Sine}_\beta[0,\gl]\geq n\big)\leq e^{-(\beta/2) n^2\ln(n/\lambda)+c n\ln(n+1)\ln(n/\lambda)+c n^2}.$$
	One has to track how the constant $c$ in their formula depends on $\gl_0$. Looking at their proof of the upper bound, we see that their argument remains valid for arbitrarily large $\gl \geq 1$ and for $n$ large enough with respect to $\gl$, since the constants involved in their recursion depend only on $\beta$. Moreover, when $n \geq C(\beta) \gl$, the inequality 
$$c n \ln (n+1) \ln (n/\gl) + c n^2 -(\beta/4) n^2 \ln (n/\gl) \leq 0$$ holds as long as the constant $C(\beta)$ is chosen large enough. This implies the result.
	\end{proof}

			\begin{figure}[h!]
		\includegraphics[width=8cm]{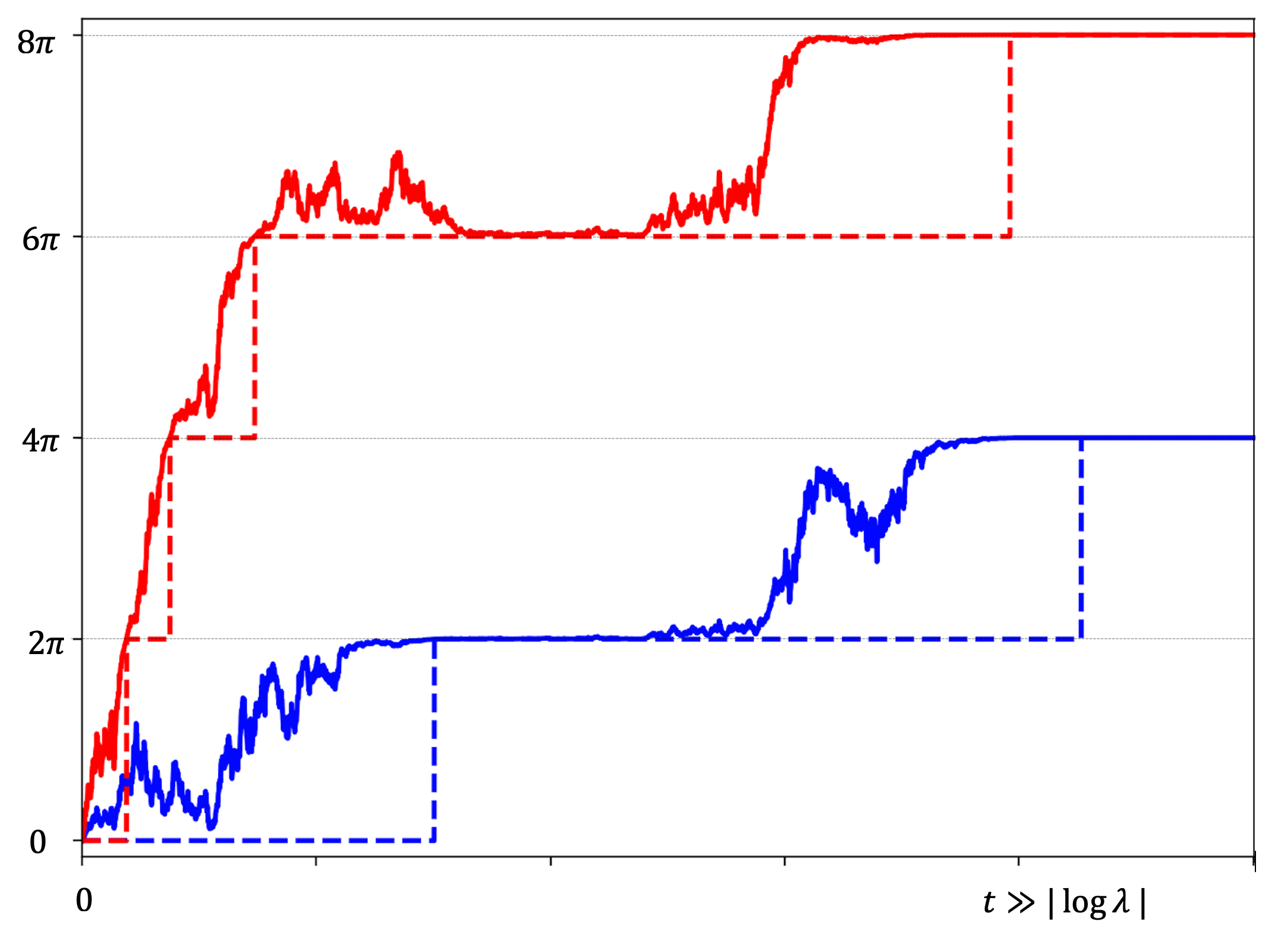}
		\caption{\small{(color online). Trajectories of the diffusions $\alpha_6$ (blue) and $\alpha_{20}$ (red) for $\beta=1$. This realization of $\text{Sine}_1$ has two points in $[0,6]$ and 4 points in $[0,20]$. Remark that the $\alpha_\lambda$ can not go below a level $2k\pi$ once it has been crossed, and that the dynamics freeze after times much larger than $|\log\lambda|$.}}\label{fig:carousel}
	\end{figure}
	
	\vspace{0,3cm}
	Let us now rephrase Theorem \ref{theo:main2} with the stochastic sine equations. Let us fix the parameters $\beta$, $\gl_0$, $k$. Define $J_1 := \{1,\cdots,k_0\}$ and $J_2 := \{k_0+1,\cdots,k\}$ and the two products:
	\begin{align}\label{defclusters}
		\mathcal{I}_1 := \prod_{i \in J_1} I_i, \quad\mbox{and}\quad  \mathcal{I}_2 := \prod_{i \in J_2} I_i,\qquad \mbox{where }I_i :=[x_i, x_i + \gl_i]\,.
	\end{align}
Moreover, the interval lengths satisfy $\gl_i\leq \lambda_0$. In the first cluster $1 \leq i \leq k_0$, we choose $x_i$ so that the intervals $I_i$ are pairwise disjoint and included in $\R^-$. In the second cluster $k_0 +1 \leq i \leq k$, we choose $x_i$ such that the intervals are all disjoint and included in $\R^+$.
	
	\medskip
	
	We define
	\begin{align*}
		\alpha^{(i)}(t) \defeq \alpha_{x_i +\lambda_i}(t)-\alpha_{x_i}(t)\,, \qquad    i \in J_1\,,
	\end{align*}
	and 
	\begin{align*}
		\alpha^{(i)}(t) \defeq \alpha_{r +x_i + \gl_i}(t)-\alpha_{r+x_i}(t)\qquad i \in J_2\,.
	\end{align*}	
	By definition, we have
	\begin{equation}
		\Big(\,\text{Sine}_\beta(I_{i_1}),\; i_1\in J_1, \quad \text{Sine}_\beta(r + I_{i_2}),\; i_2\in J_2\Big) \overset{(d)}{=}  \Big(\, \frac{\alpha^{(i)}(+\infty)}{2\pi},\; i \in \{1,\cdots,k\}\Big)\,.
	\end{equation}
	Correlations are encoded in the coupling between the two families of diffusions $i \in J_1$ and $i \in J_2$. As we will observe later, this coupling strongly depends on the diffusion $\alpha_r$
	\begin{align}
		\label{eqalphar}    \mathrm{d}\alpha_r(t) \ \ &=r\,f(t) \mathrm{d}t+ \text{Re}\left((e^{-i\alpha_r(t)}-1)\,\mathrm{d}W(t) \right)\,,
	\end{align}
	and the complex Brownian motion
	\begin{align*}
		W_r(t) := \int_0^t e^{-i\alpha_r(s)} \mathrm{d}W(s)\,.
	\end{align*}
	The diffusions $\alpha^{(i)}$ for $i \in J_1$ are given by
	\begin{align*}
		\mathrm{d}\alpha^{(i)}(t) = \gl_i f(t) \mathrm{d}t + \mbox{Re}\Big((e^{-i \alpha^{(i)}(t) }- 1) \mathrm{d}W^{(i)}(t)\Big),\qquad \mathrm{d}W^{(i)}(t) = e^{- i \alpha_{x_i}(t)} \mathrm{d}W\,.
	\end{align*}
	Similarly, the diffusions $\alpha^{(i)}$ for $i \in J_2$ are given by
	\begin{align*}
		\mathrm{d}\alpha^{(i)}(t) = \gl_i f(t) \mathrm{d}t + \mbox{Re}\Big((e^{-i \alpha^{(i)}(t) }- 1) \mathrm{d}W^{(i)}(t)\Big),\quad \ \mathrm{d}W^{(i)}(t) = e^{- i (\alpha_{x_i+r}(t) - \alpha_r(t))} \mathrm{d}W_r(t)\,.
	\end{align*}
	We denote in the following
	$$\mathbf{W}_1\defeq \big(W^{(1)},\cdots, W^{(k_0)}\big), \qquad \mathbf{W}_2\defeq \big(W^{(k_0+1)},\cdots, W^{(k)}\big).$$

	\subsection{Main steps of the proof of Theorems \ref{theo:main} and \ref{theo:main2}}
	In this section, we explain the main steps of the proof of Theorems \ref{theo:main} and \ref{theo:main2}. We fully state the key lemmas, whose proofs are deferred to the subsequent sections, and we then provide the proof of the Theorems.
	
	Define the time $T_r\defeq(4/\beta)\ln r$. This is the typical time at which our system of diffusions goes from an independent active dynamics to a dependent frozen dynamics. Namely,
	\begin{itemize}
		\item \emph{For times smaller than $T_r$} and for large values of the parameter $r$, the diffusion $\alpha_r$ defined in \eqref{eqalphar} has a strong drift and behaves almost deterministically. It follows the deterministic path given by $ t \mapsto r \int_0^t f(s) ds = r(1- e^{-\beta t/4})$. As a consequence, the complex Brownian motion $W$ is close to being independent of $W_r$. We obtain a good control of this phenomenon thanks to a discretization of the time interval.
		
		\item \emph{For times greater than $T_r$}, the diffusion $\alpha_r$ becomes more unpredictable, which effectively creates correlations. However, its impact on the correlations is controlled using that the diffusions $(\alpha^{(i)}, \; i=1,\cdots, k)$ have a slow dynamic after such time since their drift term is small.
	\end{itemize}
	Figure \ref{fig:oscillations} represents a plot of the real part of $t\mapsto e^{-i\alpha_r(t)}$ and illustrates this transition.
	\begin{figure}
		\includegraphics[width=11cm]{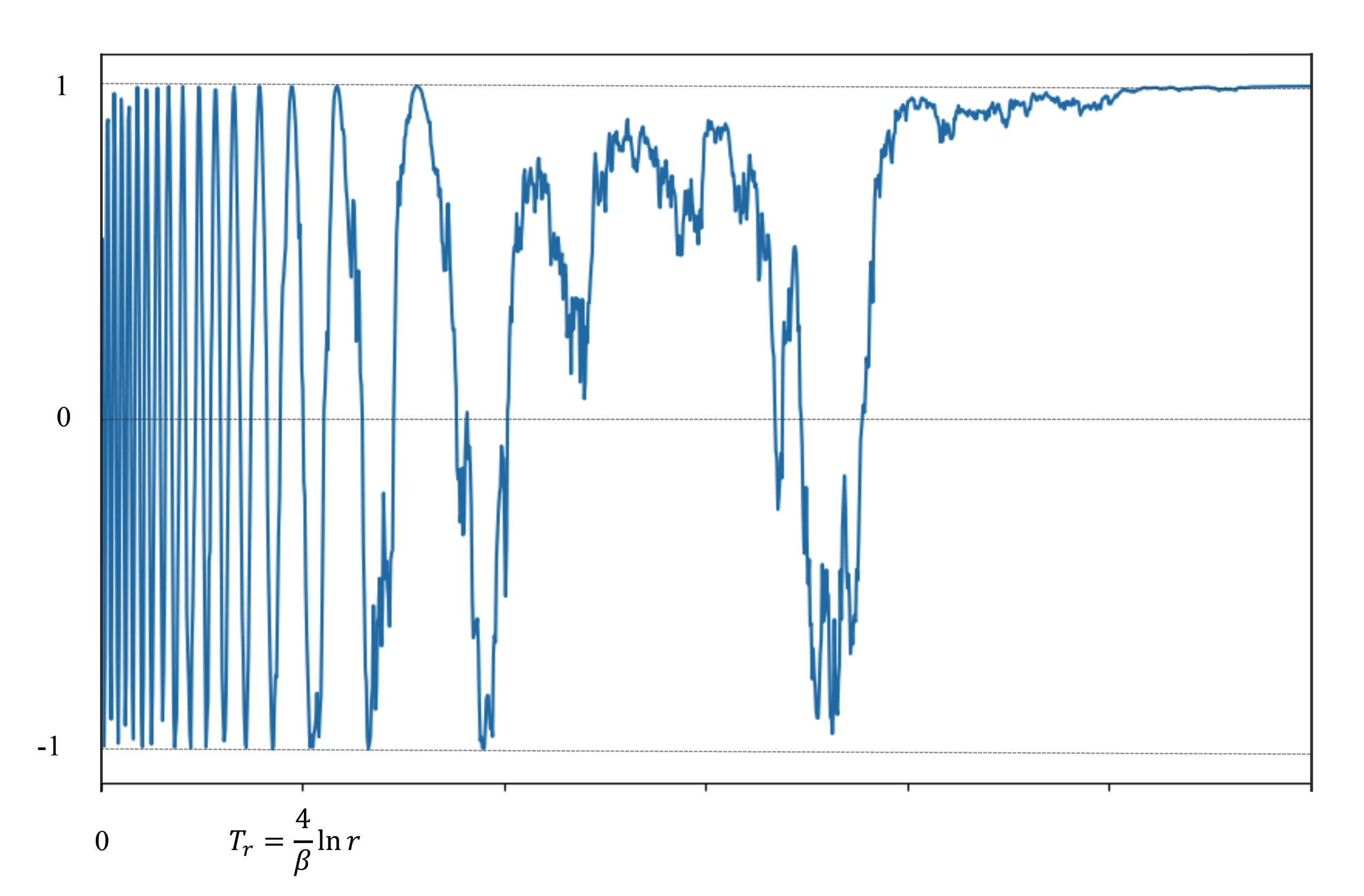}
		\caption{Simulation of the oscillations of $t\mapsto \cos(\alpha_r(t))$, $r=100, \beta=4$. At first, the dynamic is deterministic. After times of order $T_r$, randomness appears, until the system gets frozen and converges towards $1$.}\label{fig:oscillations}
	\end{figure}

	We first state a lemma showing that, after large times, the diffusion $\alpha_\gl$ is very close to its limiting value with high probability.
	\begin{Lemma}[Finite time approximation]\label{lem:finitetime}
		For all $\beta,\lambda_0>0$, there exists $C=C(\beta,\lambda_0)>0$ such that for all $\lambda \leq \lambda_0$, for all $T\geq 1$, we have 
		\begin{align*}
			\P\Big(\big|\alpha_\gl (T)  - \alpha_\gl(+\infty)\big| > \pi/2\Big) \leq C \exp\big(-(1\wedge\beta) T/128\big)\, .
		\end{align*}
	\end{Lemma}
	
	This key lemma is one of the main technical imput of the present paper. It relies on two facts: the first one is that the regions around $2\pi \N+\pi$ are unstable equilibria for the diffusion $\alpha_\gl$. Using Girsanov's theorem, we are able to prove that it escapes rapidly from these. Then, as long as the diffusion escaped from neighborhood of $2\pi \N+\pi$, the multiples of $2\pi$ are attractive.

	\medskip
	
	We then turn to the time interval where we expect that the two clusters behave in an almost independent way.

	Let us first state the lemmas for $k=2$, as they need some refinement for general $k$-point correlation functions. Take\footnote{Note that to be in the exact setting of Theorem \ref{theo:main2}, we would need to take $x_1 := - \gl$ but it does not change the bounds.} $x_1 := 0$ and $x_2 := 0$. We simply have two diffusions $\alpha^{(1)}$ and $\alpha^{(2)}$ driven by $W^{(1)} = W$ and $W^{(2)} = W_r$.

	Fix $n\geq 1$ and let $(t_j := j T/n,\; j =1,\cdots,n)$ be a discretization of the interval $[0,T]$.

	The following lemma shows that, for large times $T$ and large integer $n$, the diffusion $\alpha^{(i)}$ and a discretized version of it, measurable w.r.t.  $(W^{(i)}(t_j),\; j = 1,\cdots,n)$ are close to each other with high probability.
	
	\begin{Lemma}[Discretization]\label{lem:discrete}
		Let $\beta,\lambda_0>0$. There exists $C=C(\beta,\lambda_0)>0$ such that for all $\lambda_1,\lambda_2\leq \lambda_0$, for all $T>0$, for all integer $n$, and for $i=1,2$, there exists a piecewise-constant process $\alpha_{\mathrm{p.c.}}^{(i)}$, measurable with respect to the random vectors $(W^{(i)}(t_j),\; j = 1,\cdots,n)$, such that
		\begin{align*}
			\P\Big(\big|\alpha^{(i)}_{\mathrm{p.c.}}\big(T\big) - \alpha^{(i)}\big(T\big)\big| \geq \pi/3\Big) \leq C\frac{ e^{4T}}{\, n}\,.
		\end{align*}
	\end{Lemma}
	We then control the total variation distance between the discretizations of $(W^{(1)},W^{(2)})$ and of the independent coupling $W^{(1)}\otimes W^{(2)}$.
	\begin{Lemma}[Upper bound for total variation]\label{lem:totalvariation}
		Let $\beta>0$. There exists $C=C(\beta)>0$ such that for all $T,r>0$ and for all $n\in\mathbb{N}$, we have 
		\begin{align*}
			\mathrm{d}_{TV}\bigg(\Big(W^{(1)}(t_j),W^{(2)}(t_j)\Big)_{j=1}^n,\Big(W^{(1)}(t_j)\otimes W^{(2)}(t_j)\Big)_{j=1}^n\bigg)\leq C \frac{n^{3/2} e^{\frac{\beta}{4}T}}{r}\,.
		\end{align*}
	\end{Lemma}
	Note that the total variation distance is well-controlled for times $T$ smaller than $T_r$, which is expected since correlations appear after $T_r$. 
	
	\medskip
	
	When extending these results to general $k$-point correlation functions, we need the following refinement of Lemmas \ref{lem:discrete} and \ref{lem:totalvariation}.
	\begin{Lemma}\label{lem:kpoints}
		Let $\beta, \lambda_0>0$. There exists a constant $C=C(\beta,\lambda_0)>0$ such that the following holds. Let $k, n \geq 1$ be integers, $1 \leq k_0 < k$ and recall that $J_1 = \{1,\cdots,k_0\}$ and $J_2 = \{k_0+1,\cdots,k\}$.
		For all $T$, $r >0$ and all $\gl_i \leq \gl_0$ and $x_i$ chosen as in \eqref{defclusters}, one can construct random variables $\boldsymbol{\tilde{W}}_1\in \mathbb{C}^{k_0 \times n}$ and $\boldsymbol{\tilde{W}}_2 \in \mathbb{C}^{(k-k_0) \times n}$ on the same probability space as the original diffusions $(\alpha^{(i)})_{1 \leq i \leq k}$ as well as piecewise constant processes $(\alpha^{(i)}_{\mathrm{p.c.}})_{1 \leq i \leq k}$ such that
		\begin{itemize}
			\item For $i \in J_1$, $(\alpha^{(i)}_{\mathrm{p.c.}})_{i \in J_1}$ is measurable w.r.t. $\boldsymbol{\tilde{W}}_1$ and for $i \in J_2$, $(\alpha^{(i)}_{\mathrm{p.c.}})_{i \in J_2}$ is measurable w.r.t. $\boldsymbol{\tilde{W}}_2$ ;
			\item The total variation distance satisfies 
				\begin{align*}
				\mathrm{d}_{TV}\Big(\big(\boldsymbol{\tilde{W}}_1,\boldsymbol{\tilde{W}}_2\big),\big(\tilde{\boldsymbol{W}}_1\otimes \tilde{\boldsymbol{W}}_2\big)\Big)&\leq C\,k^{3} \, \frac{n^{5/2} \; e^{\frac{\beta}{4}T}}{r}\,;
			\end{align*}
		\item The discretization error is bounded by
			\begin{align*}
			\sup_{1\leq i\leq k}\P\Big(\big|\alpha^{(i)}_{\mathrm{p.c.}}\big(T\big) - \alpha^{(i)}\big(T\big)\big| > \pi/2\Big) &\leq \ C\frac{e^{4T}}{ n}\, .
	\end{align*}
		\end{itemize}
	\end{Lemma}
	 Recall that the diffusions in the cluster $J_1$ are driven by $\boldsymbol{W}_1$ and in the cluster $J_2$ by $\boldsymbol{W}_2$. The main technical issue for large $k$ is that strong correlations within each cluster can cause the covariance matrices of the random vectors $(W^{(i)}(t_j),\; j = 1,\cdots,n,\;i \in J_\ell)$ for $\ell =1$ or $2$ to become nearly singular. Because bounds on the total variation distance typically depend on the inverse of these covariance matrices, such near-degeneracies cause standard decoupling bounds to blow up, counterbalancing the polynomial decay induced by the oscillations of $\alpha_r$. To prevent this, we use another approximation of the continuous processes
	$\boldsymbol{W}_1$ and $\boldsymbol{W}_2$ using a spectral regularization technique. This regularization ensures that the total variation distance remains controlled while still providing a precise discretization of the diffusions.
	
	\medskip
	
	\subsection{Proof of Theorems \ref{theo:main} and \ref{theo:main2}}
	
	We focus on the proof of Theorem  \ref{theo:main2}. Indeed, applying it with $k=2$ taking the intervals $[0,1]$ and $[r, r+\gl]$ and using translation invariance we immediately obtain the result of Theorem \ref{theo:main}.
	
	\medskip
	
	Recall that $\mathcal{I}_1 := \prod_{i=1}^{k_0} I_i$ and $\mathcal{I}_2 := \prod_{i=k_0 +1}^{k} I_i$ and that $J_1=\{1,\cdots,k_0\}$ and $J_2=\{k_0+1,\cdots,k\}$, and introduce the products
	\begin{align*}
		\cP_1 := \prod_{i=1}^{k_0} \mbox{Sine}_\beta(I_i),\qquad \cP_2 := \prod_{i=k_0+1}^{k} \mbox{Sine}_\beta(r+I_i)\,.
	\end{align*}
	We need to bound from above
	\begin{align*}
		\Big|	\E\big[\cP_1 \times \cP_2\big] - \E\big[\cP_1\big]\,\E\big[\cP_2\big]\Big|\,.
	\end{align*}
	The first step is to introduce a cutoff on the products in the expectations and decompose it into probabilities as 
	\begin{align}
		&\Big|\E\big[(\cP_1 \times \cP_2) \mathbf{1}_{\{\cP_1 \leq m,\; \cP_2 \leq m\}}\big] - \E\big[\cP_1 \mathbf{1}_{\{\cP_1 \leq m\}}\big]\,\E\big[\cP_2 \mathbf{1}_{\{\cP_2 \leq m\}}\big]\Big|\,, \notag \\
		\leq &\  m^4 \sup_{p_1 \leq m, \; p_2 \leq m} \Big| \P\big[\cP_1 = p_1, \cP_2 = p_2\big] - \P\big[\cP_1 = p_1\big] \P\big[ \cP_2 = p_2\big]\Big|\,. \label{ineqP1P2small}
	\end{align}
	
	Recall that we have $ \mbox{Sine}_\beta(I_i) := \alpha^{(i)}(+\infty)/2\pi$ for $i \in J_1$, and $ \mbox{Sine}_\beta(I_i+r) := \alpha^{(i)}(+\infty)/2\pi$ for $i \in J_2$.
	
	In the products $\cP_1$ and $\cP_2$, up to a small error that we control, one can replace the values $\alpha^{(i)}(+\infty)$ by their approximations $\alpha^{(i)}_{\text{p.c.}}(T)$. More precisely, if $\lfloor x \rceil$ denotes the smallest closest integer of $x$, we use $\lfloor \alpha^{(i)}_{\text{p.c.}}(T)/(2\pi) \rceil$. As soon as $|\alpha^{(i)}_{\text{p.c.}}(T)- \alpha^{(i)}(+\infty)| < \pi$, the two values coincide since  $\alpha^{(i)}(+\infty)/(2\pi)$ is an integer.
	Let us denote the new products for $q=1,2$ 
	\begin{align*}
		\hat{\cP}_q := \prod_{i \in J_q} \big\lfloor \alpha^{(i)}_{\text{p.c.}}(T) \big\rceil\,.
	\end{align*}
	For any events $\mathcal{A}$, $Z$, $X$, $Y$, we have the inequalities
	\begin{align*}
		\big|\P[Z] -\P[X] \P[Y] \big| \leq	\big|\P[Z,\; \mathcal{A}] -\P[X,\; \mathcal{A}] \P[Y,\; \mathcal{A}] \big| + 2\, \P[\mathcal{A}^c]\,,
	\end{align*}
	and
	\begin{align*}
		\big|\P[Z,\;\mathcal{A}] -\P[X,\; \mathcal{A}] \P[Y,\; \mathcal{A}] \big| \leq	\big|\P[Z] -\P[X] \P[Y] \big| + 2\, \P[\mathcal{A}^c]\,.
	\end{align*}
	Applied to the probabilities in \eqref{ineqP1P2small}, it gives
	\begin{gather*}
		\Big| \P\big[\cP_1 = p_1, \cP_2 = p_2\big] - \P\big[\cP_1 = p_1\big] \P\big[ \cP_2 = p_2\big]\Big|\\
		\leq \Big| \P\big[\hat{\cP}_1 = p_1, \hat{\cP}_2 = p_2\big] - \P\big[\hat{\cP}_1 = p_1\big] \P\big[ \hat{\cP_2} = p_2\big]\Big|  + 4 \; \P\big[\cup_{i=1}^k\{|\alpha^{(i)}_{\text{p.c.}}(T)- \alpha^{(i)}(\infty)| \geq \pi\}\big].
	\end{gather*}
	We bound the first term by the total variation distance between the $n \times k$ dimensional complex vectors $(\tilde{\textbf{W}}_1, \tilde{\textbf{W}}_2)$ and $(\tilde{\textbf{W}}_1\otimes\tilde{\textbf{W}}_2)$ of Lemma \ref{lem:kpoints}
	\begin{align*}
		\Big| \P\big[\hat{\cP}_1 =p_1,\; \hat{\cP}_2 = p_2\big] - \P\big[\hat{\cP}_1 = p_1\big] \P\big[ \hat{\cP_2} = p_2\big]\Big| \leq \mbox{d}_{TV}\Big(\big(\tilde{\textbf{W}}_1, \tilde{\textbf{W}}_2\big),\big(\tilde{\textbf{W}}_1\otimes\tilde{\textbf{W}}_2\big)\Big)\,.
	\end{align*}
	We use a union bound on the second term and split each error term into two parts
	\begin{align*}
		\P\Big[\big|\alpha^{(i)}_{\text{p.c.}}(T)- \alpha^{(i)}(\infty)\big| \geq \pi\Big] \leq \P\Big[\big|\alpha^{(i)}_{\text{p.c.}}(T)- \alpha^{(i)}(T)\big| \geq \frac{\pi}{2}\Big] + \P\Big[\big|\alpha^{(i)}(T)- \alpha^{(i)}(\infty)\big| \geq \frac{\pi}{2}\Big],
	\end{align*}
	which are controlled thanks to Lemma \ref{lem:kpoints} and Lemma \ref{lem:finitetime}.
	By choosing the parameters as $T=\frac{4}{41+\beta}\log r$ and $n=e^{(16-\beta) T/14} r^{2/7}$ (a slightly suboptimal, yet asymptotically tight choice), we obtain the bound
	\begin{align}
		&\Big|\,\E\Big[(\cP_1 \times \cP_2) 1_{\{\cP_1 \leq m,\; \cP_2 \leq m\}}\Big] - \E\Big[\cP_1 1_{\{\cP_1 \leq m\}}\Big]\E\Big[\cP_2 1_{\{\cP_1 \leq m\}}\Big]\Big|\, \leq C\, m^4\, k^{3} r^{-c_\beta}\,, \label{ineqP1P2petit}
	\end{align}
	for $c_\beta=\frac{1\wedge\beta}{32(41+\beta)}$ and some $C := C(\beta,\gl_0)$. In the following, we will use again $C$ for some constant depending only on $\beta$ and $\gl_0$ that may change from line to line.
	
	\medskip
	
	It remains to bound:
	\vspace{0,1cm}
	\begin{align*}
		\E\Big[(\cP_1 \times \cP_2) \mathbf{1}_{\{\cP_1 \geq m\}}\Big],\qquad 	\E\Big[\cP_1\,\mathbf{1}_{\{\cP_1 \geq m\}}\Big]\E\big[\cP_2\big]\,,
	\end{align*}
	and the symmetric terms obtained when exchanging the index $1$ and $2$. We have two options to bound these expectations and we use both of them to deal with the first term, as the same calculations are valid for the other terms.
	
	\medskip
	
	\emph{In the first approach,} we use the fact that the intervals composing $\mathcal{I}_1$ and $\mathcal{I}_2$ are non-overlapping, and therefore, we have a good control on the total number of points in the union over all intervals in terms of the diameter of this union:
	\begin{align*}
		\cP_1 &\leq \exp \Big( \sum_{i=1}^{k_0} \mbox{Sine}_\beta(I_i) - k_0\Big) \leq \exp \Big(\mbox{Sine}_\beta\big([-L,0]\big)\Big)\,.
	\end{align*}
	It gives
	\begin{align*}
		\P\big[\cP_1 \geq m\big] &\leq \P\Big[\mbox{Sine}_\beta\big([-L,0]\big) \geq \ln m\Big]\\
		&\leq C \exp\Big(- (\beta/4) \big(\ln m \big)^2 \ln \Big(\frac{\ln m}{L}\Big)\Big)\,,\\
		&\leq C \exp\big(- (\beta/4) (\ln m)^2 \big)\,,
	\end{align*}
	valid when $m \geq \exp(C L)$ for $C:= C(\beta,\gl)$, chosen greater than $e$, using the overcrowding estimate of Theorem \ref{overcrowdingTh}. Moreover,
	\begin{align*}
		\E\big[(\cP_1)^4\big] &\leq  \E\big[\mathbf{1}_{\{\mathcal{P}_1\leq e^{C(L+1)}\}}(\cP_1)^4 \big]+\sum_{p = \lfloor e^{ C(L+1)} \rfloor}^{\infty} \P\big[\cP_1 \geq p^{1/4}\big] \\
		&\leq  e^{4C (L+1)} + \sum_{p = \lfloor e^{C(L+1)} \rfloor}^{\infty} C \exp\big(- \frac{\beta}{64} \big(\ln p \big)^2 \big)\,.
	\end{align*}
By increasing $C$ if necessary, one may suppose that each integer $p$ in the series satisfies $\beta\ln p/64\geq 2$, which implies:
	\begin{align*}
		C\exp(C L)+C\sum_{p=\lfloor e^{CL}\rfloor}^\infty p^{-2} \leq\ C \exp(C L)\,.
	\end{align*}
	In the same way, $	\E\big[(\cP_2)^4\big] \leq  C \exp(C L)$.
	
	\medskip
	
	In total we obtain thanks to H\"older inequality:
	\begin{align}
		\E\Big[(\cP_1 \times \cP_2) 1_{\{\cP_1 \geq m\}}\Big] &\leq 	\E\big[(\cP_1)^4\big]^{1/4} \times \E\big[(\cP_2)^4\big]^{1/4}\times \P\big[\cP_1 \geq m\big]^{1/2} \notag \\
		&\leq C \exp(C L)  \exp\big(- (\beta/8) (\ln m)^2\big)\,. \label{ineqP1P2grands}
	\end{align}
	We take $\ln m = c_\beta\ln r/8 +C(L+1)$. Up to increasing $C$, we may suppose $C\beta \geq 16$, so that $\beta(\ln m)^2/8\geq c_\beta \log r$. The last term \eqref{ineqP1P2grands} is smaller than
	$$C \exp(C L)  r^{-c_\beta }.$$
	Replacing $m$ with its value in \eqref{ineqP1P2petit} yields the bound with the pre-factor $K_1(L,k):= C e^{CL}k^3$ multiplied by $1/r$ at a power slightly smaller $c_\beta$, say $c_\beta/2$.
	
	\medskip
	
	\emph{In the second approach,} we use H\"older's inequality to write
	\begin{align*}
		\E\big[(\cP_1)^2\big] \leq \E\Big[\mbox{Sine}_\beta\big([0,\gl_0]\big)^{2k_0}\Big] \leq C^{2k_0} (2k_0)^{k_0} \,.
	\end{align*}
	Therefore, using $\E\big[(\cP_1)^2 \; 1_{\{\cP_1 \geq m\}}\big] \leq \E\big[(\cP_1)^2\big]/m^2$ and the fact that $k_0^{k_0}(k-k_0)^{k-k_0}\leq k^k$, we obtain
	\begin{align*}
		\E\big[(\cP_1\times \; \cP_2) \; 1_{\{\cP_1 \geq m\}}\big] &\leq \frac{C^{k} k^{k/2}}{m}\,.
	\end{align*}
	Choosing $m = r^{c_\beta/5}$, we obtain the result with $K(\infty,k)$ and $r^{-c_\beta/5}$. 
	\begin{rem}[On the absolute power constant]
		Remark that the constant $c$ of Theorem \ref{theo:main2} can be taken equal to $2^{-13}$. If one want to find the optimal value when $\beta$ is large, we  can get $1/32$ with the first approach.
	\end{rem}

	\subsection{Correlation functions of point processes and proof of Theorem \ref{cor:truncated}.}\label{subsec:correlations}

	Let us recall some basic facts about real-valued point processes that one can find for example in Lenard \cite{lenard_correlation_1973}. A \textit{point configuration} $\gamma$ is a integer-valued Radon measure (which can be identified with a locally finite multiset of points in $\mathbb{R}$).
	We denote by $\mathrm{Conf}(\mathbb{R})$ the space of point configurations. A point configuration $\gamma$ is said to be \textit{simple} if $\gamma(\{x\}) \le 1$ for any $x \in \mathbb{R}$. The space $\mathrm{Conf}(\mathbb{R})$ is Polish when equipped with the topology of \textit{vague convergence} (see e.g.  \cite[Section 15.7]{kallenberg1983}). A sequence $(\gamma_n)_{n\geq1}$ converges vaguely to $\gamma$ if and only if for all continuous compactly supported functions $f:\mathbb{R}\to\mathbb{R}$,
	$$
	\sum_{u\in\gamma_n}f(u) \underset{n\to\infty}{\longrightarrow} \sum_{u\in\gamma}f(u)\,.
	$$
	We endow $\mathrm{Conf}(\mathbb{R})$ with its Borel $\sigma$-algebra. A simple point process $\Gamma$ is  by definition a random variable taking values in $\mathrm{Conf}(\mathbb{R})$ such that $\Gamma$ is simple almost surely. The law of $\Gamma$ is uniquely determined by the collection of Laplace functionals
	$$ \mathbb{E}\Big[\exp\Big(-\sum_{u \in \Gamma} f(u)\Big)\Big]$$
	evaluated on all non-negative continuous compactly supported functions $f$.
	
	The \emph{joint intensity measures} (also called factorial moment measures) of a simple point process $\Gamma$, when they exist, are measures $\rho^{(k)}$ on $\mathbb{R}^k$ defined for any Borel set $B \subset \mathbb{R}^k$ by
	\begin{align*}
		\rho^{(k)}(B) = \mathbb{E}\Big[ \sum_{(x_1, \dots, x_k) \in \Gamma^{\wedge k}} \mathbf{1}_B(x_1, \dots, x_k) \Big] \,,
	\end{align*}
	where $\Gamma^{\wedge k} := \{(x_1, \dots, x_k)\in \Gamma^k,\; x_i \neq x_j \}$ denotes the set of $k$-tuples of \emph{distinct} points of $\Gamma$.
	
	When $B = D_1 \times \cdots \times D_k$ with $D_1, \dots, D_k$ being mutually disjoint Borel subsets of $\mathbb{R}$, this simplifies to
	\begin{align*}
		\rho^{(k)}(D_1 \times \cdots \times D_k) = \mathbb{E}\Big[\prod_{i=1}^k \Gamma(D_i)\Big]\,,
	\end{align*}
	where $\Gamma(D_i)$ denotes the number of points of $\Gamma$ in $D_i$. When one allows for overlaps between the sets $D_j$, the situation is more delicate. For example, taking $B = D^k$ yields the $k$-th factorial moment
	\begin{align*}
		\rho^{(k)}(D^k) = \mathbb{E}\Big[\frac{\Gamma(D)!}{(\Gamma(D) - k)!}\Big] \,.
	\end{align*}
	If $\Gamma(D)$ has exponential tails for all bounded Borel sets $D$, then the measures $\rho^{(k)}$ exist for all $k$. Moreover, under this condition, the family $(\rho^{(k)})_{k \geq 1}$ uniquely determines the law of $\Gamma$: the moments $m_p\defeq \mathbb{E}[\Gamma(D)^p]$  characterize uniquely the law of $\Gamma(D)$ as they grow slower than $(C\,p)^p$, therefore satisfy Carleman's condition 
	$$\sum_{p\geq 1}(m_{2p})^{1/(2p)}=+\infty.$$
	
	If the measures $\rho^{(k)}$ are absolutely continuous with respect to the Lebesgue measure on $\mathbb{R}^k$, we say that the point process admits \emph{correlation functions} (or joint intensities). With an abuse of notation, we denote their densities again by $\rho^{(k)}(x_1,\cdots,x_k)$.
	
	The function $\rho^{(1)}$ is called the \textit{density of states} (or simply intensity). When the point process is translation-invariant, $\rho^{(1)}$ is constant. More generally, translation-invariance implies that $\rho^{(k)}(x_1 + r, \dots, x_k +r)= \rho^{(k)}(x_1, \dots, x_k)$ for all $r \in \R$.
	
	Correlation functions are crucial for computing several interesting observables. For instance, the discrepancy of a bounded Borel set $B$ is given by:
	\begin{align*}
		\text{Disc}(B) &\coloneqq \text{Var}(\Gamma(B)) \\
		&= \int_{B \times B} \big(\rho^{(2)}(x,y)-\rho^{(1)}(x)\rho^{(1)}(y)\big)\mathrm{d}x\mathrm{d}y + \int_B \rho^{(1)}(x)\mathrm{d}x.
	\end{align*}
	
	The existence of correlation functions (as densities, opposed to measures) is a non-trivial result. For the Sine$_\beta$ point process, the existence of the measures $\rho^{(k)}$ is guaranteed by the exponential tails of Sine$_\beta(B)$ for any bounded Borel set $B$ (see Point 4 in Subsection \ref{subsec:carousel}). Furthermore, one can prove that they admit densities for example using the DLR equations obtained in \cite{dereudre_dlr_2021}, we refer to \cite{Malvy_these} for more details.

	\medskip

	We have defined in this paper two quantities to control the behavior of correlation functions when some points are widely separated. The first is the (two clusters) \emph{partially truncated correlation function}:
		$$\rho^{(k)}(x_1,\cdots,x_k)-\rho^{(k_0)}(x_1,\cdots,x_{k_0})\rho^{(k-k_0)}(x_{k_0+1},\cdots, x_k)\,.$$
	The second is the \emph{$k$-point truncated correlation functions}. As mentioned above, they are defined recursively as follows: $\rho_T^{(1)}(x_1) = \rho^{(1)}(x_1)$ and
	\begin{align*}
		\rho_T^{(k)}(x_1,\cdots,x_k)=\rho^{(k)}(x_1,\cdots,x_k)-\sum_{\pi = (J_1, \cdots, J_j),\; \pi \neq \{1,\cdots,k\}}\rho^{(|J_1|)}_T(X_{J_1})\cdots\rho_T^{(|J_j|)}(X_{J_j})\,,
	\end{align*}
	where the sum runs over all the partitions $\pi = \{J_1, \cdots, J_j\}$ of $\{1,\cdots,k\}$ except the trivial one $\{1,\cdots,k\}$ and $X_{J}:=(x_{i})_{i \in J}$. 
	
	Using the Möbius inversion formula on the partially ordered set of partitions of $\{1,\cdots,k\}$, one can express $\rho^{(k)}_T$ in terms of the $\rho^{(i)}$'s as
	\begin{equation}
		\rho_T^{(k)}(x_1,\cdots,x_k)=\sum_{j=1}^k(-1)^{j-1}(j-1)!\sum_{\pi = (J_1, \cdots, J_j)}\rho^{(|J_1|)}(X_{J_1})\cdots \rho^{(|J_j|)}(X_{J_j})\,,\label{RhoTasrho}
	\end{equation}
	where the second sum runs over all partitions $\pi$ with exactly $j$ blocks. 

	Such combinatorial identities are closely related to the Stirling numbers of the second kind $\stirling{k}{j}$, which counts the number of partitions of a $k$-element set into $j$ non-empty blocks, and the ordered Bell numbers $\text{Bell}(k)$ which count the number of ordered partitions of a $k$-element set. They are defined as
	$$\text{Bell}(k)=\sum_{j=0}^k\stirling{k}{j}j!\,.$$
	As $k\to+\infty$, the ordered Bell numbers admit the asymptotic $\text{Bell}(k)\sim \frac{k!}{2(\ln 2)^{k+1}}$, see e.g. \cite{zhang2025explicit} for a detailed overview over cumulants and combinatorics on partitions. 
	
	\medskip
	
	A uniform bound on the partially truncated correlation functions yields control over fully truncated correlation functions, up to a good control on the integrated correlation functions and the bound
	\begin{equation}\label{eq:combinatorics}\sum_{\pi=(J_1,\cdots, J_j)}(j-1)!= \sum_{j=0}^{k}\stirling{k}{j}(j-1)!\leq \text{Bell}(k),\end{equation}
	where the sum runs over all the partitions of $\{1,\cdots,k\}$ (including the trivial partition equal to $\{1,\cdots,k\}$). One can show that the sum on the left hand side is equal to $2\text{Bell}(k-1)$.
	We provide in the rest of the section a proof of the decay of truncated correlation functions. More precisely, we show that Theorem \ref{cor:truncated} is a consequence of Theorem \ref{theo:main2}. 
	
	\begin{proof}
		Consider the intervals $I_i$, $i=1,\cdots,k$. Let $S_1 \sqcup S_2 = \{1,\cdots,k\}$ be a partition of the indices such that the distance between the clusters $A := \cup_{i \in S_1} I_i$ and $B = \cup_{i \in S_2} I_i$ is maximized.
		Recall that we have a super-multiplicative bound on the (integrated) $k$-point correlation function of the form
		\begin{align*}
			\int_{\prod_{j=1}^k I_j}\rho^{(k)}(\textbf{x}) \mathrm{d}\textbf{x} \leq G_k\,,\qquad G_k := C^k k^{k/2}\,,
		\end{align*}
		satisfying $G_{k}G_{k'} \leq G_{k+k'}$. Up to changing the constant $C=C(\beta,\lambda_0)>0$,
		we have the rough \emph{a priori} bound
		\begin{align}\label{eq:rough bound}
			u_k := \int_{\prod_{i=1}^k I_i}	|\rho^{(k)}_T(\textbf{x})|\mathrm{d}\textbf{x} \leq k!\; G_k\,.
		\end{align}
		Indeed, using \eqref{RhoTasrho} and \eqref{eq:combinatorics}, we have
		\begin{align*}
			u_k &\leq  \sum_{\pi=(J_1,\cdots, J_j)}(j-1)!\prod_{l=1}^j\Big(\int_{J_l}\rho^{(|J_l|)}(\mathbf{x}_{J_l})\mathrm{d}\mathbf{x}_{J_l}\Big)\,\leq \ \text{Bell}(k)\,G_k\sim \frac{k!}{(\ln 2)^k}\,G_k.
		\end{align*}
		Let us denote by 
		\begin{align*}
			\mathcal{E}_k := \int_{\textbf{x}_1\in\prod_{i\in S_1}I_i} \int_{\textbf{x}_2\in\prod_{i\in S_2}I_i} 	\big|	\rho^{(k)}(\textbf{x}_1,\textbf{x}_2) - \rho^{(k_0)}(\textbf{x}_1) \rho^{k-k_0}(\textbf{x}_2)\big|\mathrm{d}\textbf{x}
		\end{align*}
		the decoupling of clusters associated to $S_1$ and $S_2$, and thus to $A$ and $B$.
		We write the difference as
		\begin{align*}
			\rho^{(k)}(\textbf{x}_1,\textbf{x}_2) - \rho^{(k_0)}(\textbf{x}_1) \rho^{(k-k_0)}(\textbf{x}_2) = \sum_{\pi \in \Pi(A,B)} \prod_{S \in \pi} \rho_T^{(|S|)}(\mathbf{x}_S)\,,
		\end{align*}
		where $\Pi(A,B)$ denotes the set of partitions $\pi$ connecting $A$ and $B$ that is, partitions containing at least one block $S$ such that the vector $\boldsymbol{x}_S\defeq (x_i)_{i\in S}$ has coordinates in both $A$ and $B$. Isolating the trivial partition $\{1,\cdots,k\}$, we obtain
		\begin{align*}
			\rho_T^{(k)}(\textbf{x}) = \rho^{(k)}(\textbf{x}_1,\textbf{x}_2) - \rho^{(k_0)}(\textbf{x}_1) \rho^{(k-k_0)}(\textbf{x}_2) - \sum_{\pi \in \Pi(A,B),\; \pi \neq \{1,\cdots,k\}} \prod_{S \in \pi} \rho_T^{(|S|)}(\boldsymbol{x}_S)\,.
		\end{align*}
		For each $\pi \in \Pi(A,B)$, $\pi \neq \{1,\cdots,k\}$, we identify a block $S(A,B) \in \pi$ connecting $A$ and $B$ which is of length $s \in \{2,\cdots,k-1\}$. We will first sum on all the valid blocks $S(A,B)$ of length $s$ connecting $A$ and $B$. We have $ \mathcal{S}(k_0, k-k_0, s) := \binom{k}{s} - \binom{k_0}{s} - \binom{k-k_0}{s}$ possible choices for this set. Then for any fixed $S(A,B)$, we bound the remaining product over the other (integrated) terms of the partition $\prod_{S \in (\pi \backslash S(A,B))} \rho^{(|S|)}_T(\boldsymbol{x}_S)$, using \eqref{eq:rough bound}. This gives
		%
		\begin{align}\label{recuk}
			u_k \leq \mathcal{E}_k  + \sum_{s=2}^{k-1} \mathcal{S}(k_0,k-k_0,s)\, u_s\, G_{k-s}\,.
		\end{align}
		Thanks to Theorem \ref{theo:main2}, we have the upper bound
		\begin{align*}
			\mathcal{E}_k \leq C^k k^{k/2} d(A,B)^{-c \frac{\beta}{1+ \beta^2}}\,.
		\end{align*}
		One can check by induction that together with \eqref{recuk}, we obtain 
		\begin{align*}
			u_k \leq C^k \; k! \; k^{k/2} d(A,B)^{-c \frac{\beta}{1+ \beta^2}}\,,
		\end{align*}
		for a different $C$ (depending on $\gl_0$ and $\beta$) which gives our result. Indeed, if we suppose that the inequality holds for all $s  <k$, then 
		\begin{align*}\label{recuk}
			u_k &\leq \Big(C^k k^{k/2} + \sum_{s=2}^{k-1}  C^k\,\binom{k}{s} s! \,s^{s/2} \, (k-s)^{(k-s)/2} \Big) d(A,B)^{-c \frac{\beta}{1+ \beta^2}}\\
			&\leq C^{k} \Big(k^{k/2} + k! \sum_{s=2}^{k-1} s^{s/2} (k-s)^{(k-s)/2}\Big) d(A,B)^{-c \frac{\beta}{1+ \beta^2}}\,.
		\end{align*}
		The maximum of $s^{s/2} (k-s)^{(k-s)/2}$ is at one of the boundaries $s=2$ or $s=k-1$, for which we get respectively $2 (k-2)^{(k-2)/2}$ or $(k-1)^{(k-1)/2}$. We obtain the desired bound. 
			\end{proof}
		\begin{rem}
			Note that we used the bound $K(k,\infty)$ for the proof instead of the optimal one. Indeed, the sum over all partitions against the super-multiplicative bound $G_{k-s}$ dominates the asymptotics, leading to $k!\; G_k$ behavior in any case.
		\end{rem}

	\section{Behavior after large times, Proof of Lemma  \ref{lem:finitetime}}
\subsection{Strategy of proof} The goal of this section is to prove Lemma \ref{lem:finitetime} which states that for large time $T$, the diffusion $\alpha_\gl$ is close to its limiting value with high probability. 

For a fixed $\gl$, the diffusion $\alpha_\gl$ satisfies the following Stochastic Differential Equation (SDE)
\begin{align*}
	d \alpha_\gl(t) = \gl f(t) dt + 2 \sin(\alpha_\gl(t)/2) dB(t)\,,
\end{align*}
where $B$ is a real standard Brownian motion.


Suppose that for some large time $T$, the point $\alpha_\gl(T)$ is at some distance from its limiting value. We break this event into two possibilities about the time interval $[T/2,T]$:
\begin{itemize}
	\item  Either the diffusion $\alpha_\gl$ did not get very close to any multiple of $2\pi$ during the time interval $[T/2,T]$ ; 
	\item Or the diffusion $\alpha_\gl$ reached a value very close to some multiple of $2\pi$ during the interval $[T/2,T]$. In this case, it has to move away from this multiple afterwards. Indeed either the multiple corresponds to $\alpha_\gl(+\infty)$ and it has to be at some distance of this point at time $T$, or it does not correspond to $\alpha_\gl(+\infty)$ and $\alpha_\gl$ then has to reach another multiple of $2\pi$.
\end{itemize}
We will prove that both events have small probability. On these two events, one can hint the behavior of the diffusion $\alpha_\lambda$ using an appropriate Girsanov's change of measure, see \eqref{eq:girsanov}. In Figure \ref{fig:girsanov}, we sampled in blue the trajectories of $\alpha_\lambda$ under this change of measure, as compared to their initial trajectories in red.

		\begin{figure}[h]\centering
		\subfloat{\includegraphics[scale=0.3]{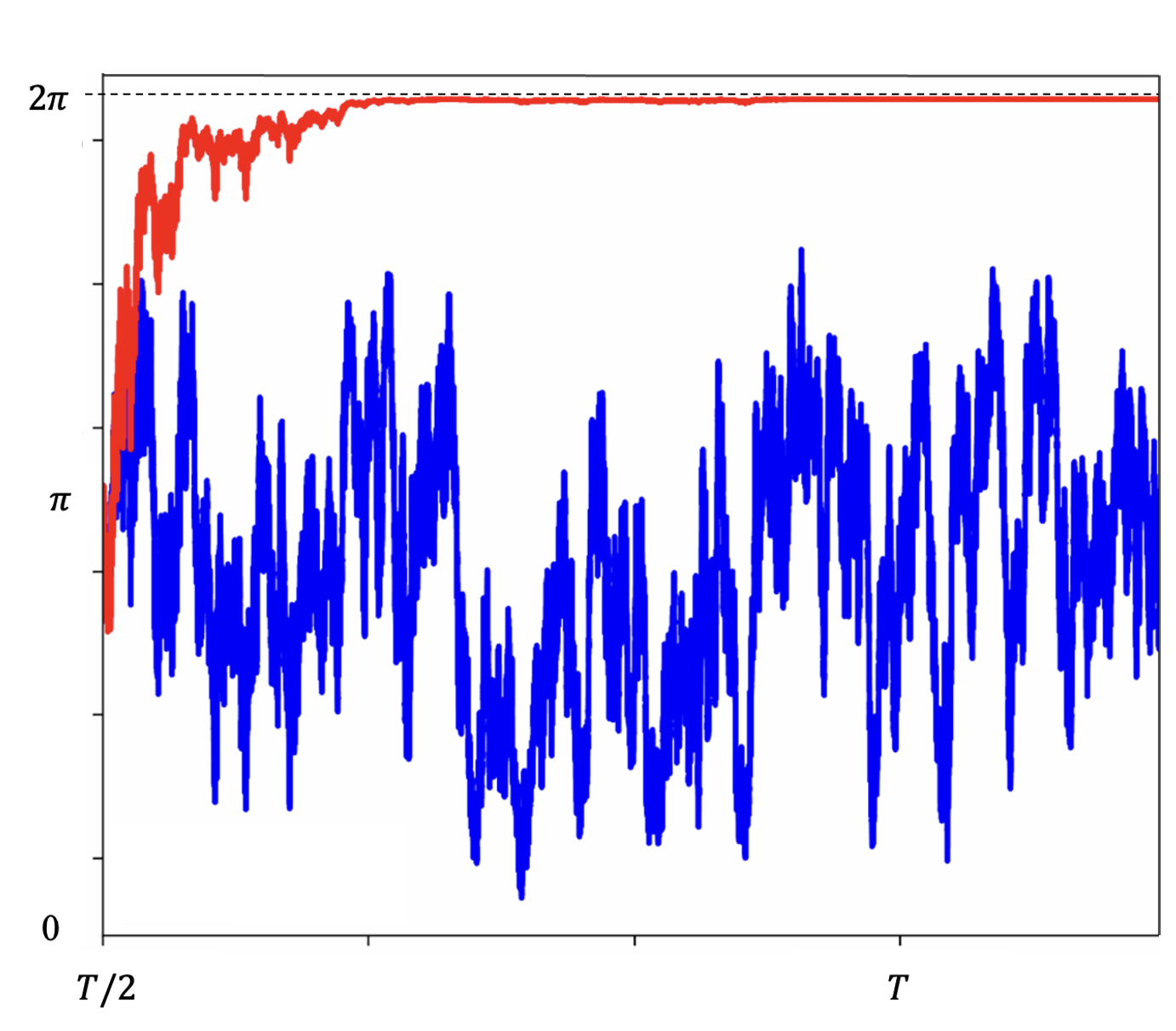}}
		\hspace{0,2cm}
		\subfloat{ \includegraphics[scale=0.3]{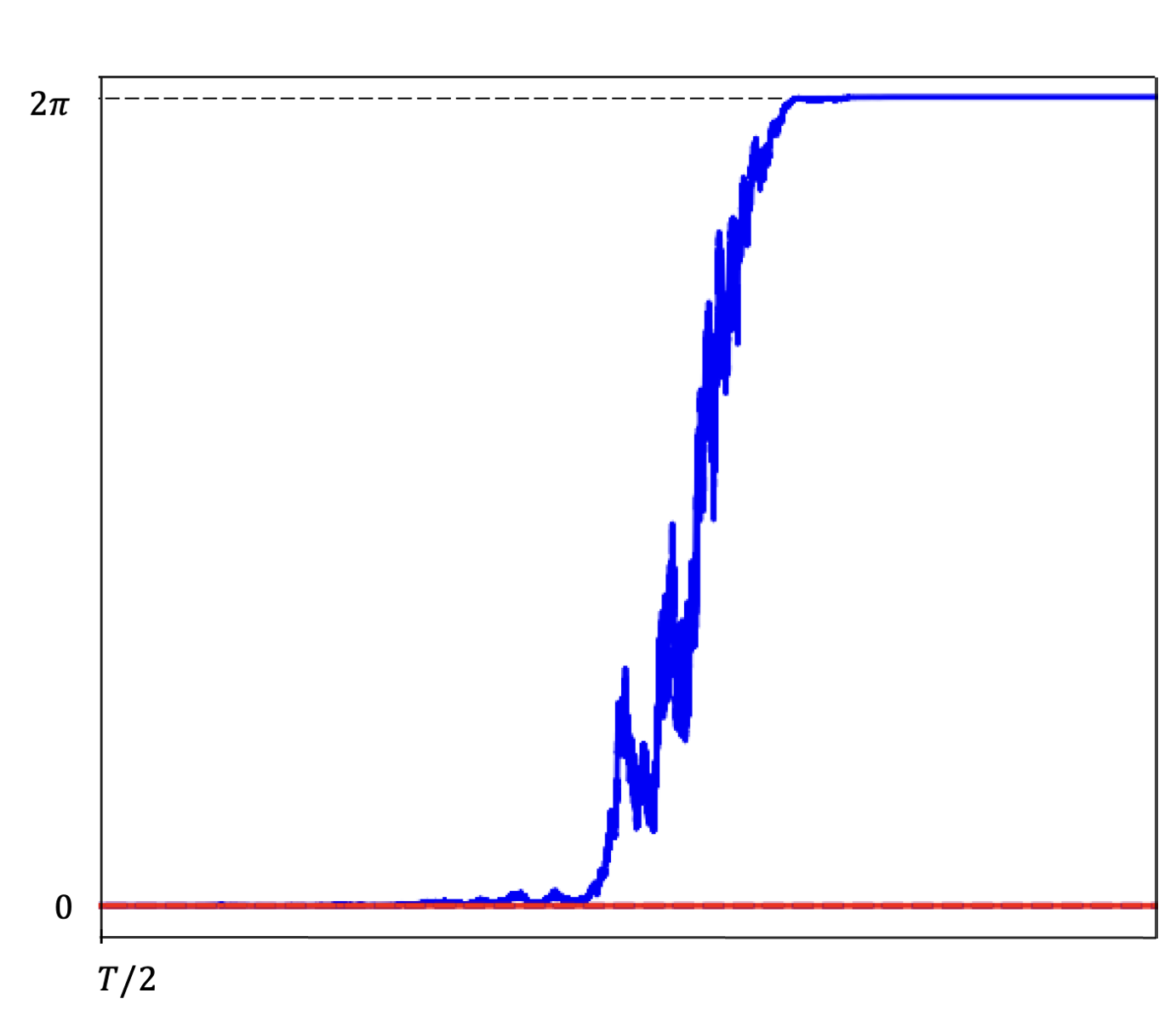}}
		\caption{\small{(color online). On the left figure, the blue path is a sample of the trajectory of $\alpha_\lambda$ started at $\pi$ under the Girsanov change of measure \eqref{eq:girsanov} with parameters $\lambda=1$, $\beta=200$. On the right figure, the blue path is a sample of the trajectory of $\alpha_\lambda$, started in $0$ at time $T/2$ ``conditioned to attain $2\pi$'' (that is with the Girsanov change of measure of  \eqref{eq:girsanov} until it reaches $\pi$ and then back to the initial measure), with $\lambda=1$, $\beta=4$. In both figures, the red path is a sample of the trajectory of $\alpha_\lambda$ without any change of measure and with the same parameters and driving Brownian motion.}}\label{fig:girsanov}
	\end{figure}
	\medskip
	To analyze the behavior of $\alpha_\gl$, it is convenient to introduce the process 
	\begin{align*}
		R := \ln(\tan(\{\alpha_\gl\}_{2\pi}/4))
	\end{align*} 
	where $\{x\}_{2\pi} = x - (2\pi)\,\lfloor x/(2\pi) \rfloor$ denotes the fractional part of $x$ modulo $2\pi$. It evolves according to the following SDE:
	\begin{align}
		\mathrm{d}R(t)=\frac{\lambda}{2} f(t)\cosh R(t) \mathrm{d}t +\frac{1}{2}\tanh R(t)\mathrm{d}t+\mathrm{d}B(t)\,.
	\end{align}

	The process $R$ starts at $-\infty$ at time $0$ and may explode to $+\infty$ at some finite time (it does so each time $\alpha_\gl$ reaches a multiple of $2\pi$). Once this happens, it immediately restarts at $-
	\infty$, which ensures that $R$ is well defined for all time $t \geq 0$. The advantage of studying $R$ instead of $\alpha_\gl$ is that it is a diffusion with a constant noise, evolving in a (non-stationary) potential. This potential has a well located around $- \beta t/4$ at time $t$ (which corresponds to the values $(2\pi \N)^+$ for $\alpha_\gl$ for large times) and an unstable equilibrium around $0$ (corresponding to $\pi$ for the diffusion $\{\alpha_\gl\}_{2\pi}$). Moreover, on the segment well included in $(-\beta t/4,0)$ (resp. $(0,\beta t/4)$), the drift is almost constant equal to $-1/2$ (resp. $+1/2$). This explains why the diffusion may converge to a multiple of $2\pi$ from below for $\beta >2$, which corresponds to a trajectory of $R$ going growing to $+\infty$ thanks to its $+1/2$ drift without reaching the slope in $(0,\beta t/4)$ (which would lead to an explosion in finite time). See Figure \ref{fig:potential} for a plot of the potential at some large time $t$.
	
		\begin{figure}[!h]
		\includegraphics[width = 7cm]{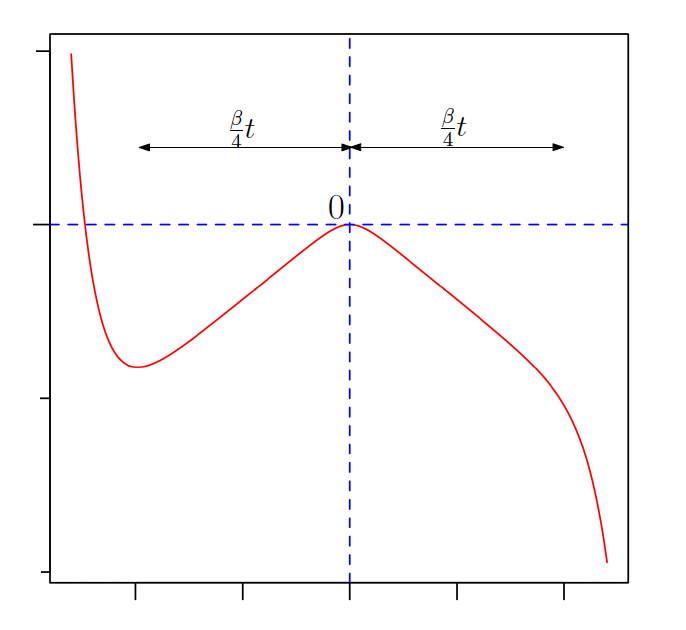}
		\caption{Plot of the potential of the diffusion $R$ at a large time $t$, for $\lambda=1$.}\label{fig:potential}
	\end{figure}
	
	\medskip
	
	Let $r_0 := (1\wedge\beta)T/64$ and define the following stopping time for $R$
	\begin{align*}
		&\sigma :=\inf\{s \geq T/2,\; |R(s)| \geq r_0\}\,,\\
		&\tau := \inf\{s \geq \sigma,\; |R(s)| \leq r_0 /2\}\,.
	\end{align*}
	We will focus on the proof in the case $\lambda\leq 1$. For the more general case $1 \leq \lambda \leq  \lambda_0$, it is sufficient to wait an extra time $\log \lambda$ to return to the case $\lambda\leq 1$.
	
	We state in the following lemma that it is unlikely for $R$ to stay away from $\pm \infty$ during a long time interval $[T/2,T]$:
	\begin{Lemma}[Repulsion of the region near $0$]\label{lem:repulsionaway0} For all $\beta,\lambda_0>0$, there exists $C=C(\beta,\lambda_0)$ such that for all $\lambda\leq \lambda_0$, $ T\geq 1$,
		\begin{align*}
			\P\big(\sigma \geq T\big) \leq C \exp\big(-(1\wedge\beta) T/128\big)\,.
		\end{align*}
	\end{Lemma}
	The proof of Lemma \ref{lem:repulsionaway0} can be found in Subsection \ref{subsec:lemrepulsion}.
	Then we state that when $R$ reaches some value near $\pm \infty$, it is difficult to leave it. 
	
	\begin{Lemma}[Attraction to the region near $\pm \infty$]\label{lem:attractioninfty} For all $\beta,\lambda_0>0$, there exists $C=C(\beta,\lambda_0)>0$ such that for all $\lambda\leq \lambda_0$, $T\geq 1$,
		\begin{align*}
			\P\big(\sigma \leq T,\; \tau < +\infty\big) \leq C \exp\big(-(1\wedge\beta) T/128\big)\,.
		\end{align*}
	\end{Lemma}
	The proof of Lemma \ref{lem:attractioninfty} can be found in Subsection \ref{subsec:prooflemattractioninfty}.
	Using the same proof, we can obtain better estimates for large $\beta$, namely $\exp(-c\beta T)$. the constant polynomial decay for large $\beta$ in Theorem \ref{theo:main2} comes from Lemma \ref{lem:repulsionaway0}.	
	\medskip
	
	Using that when $r \gg 1$, $4 \arctan (e^r) = 2\pi- 4 e^{-r} + O(e^{-3r})$ and $4 \arctan (e^{-r}) = 4 e^{-r} + O(e^{-3r})$, Lemma \ref{lem:repulsionaway0} and Lemma \ref{lem:attractioninfty} about $R$ easily translate to $\alpha_\gl$. Noting that $4e^{-r_0/2}$ is smaller than $\pi/2$ for large times $T$, Lemma \ref{lem:finitetime} is corollary of these lemmas.
	
	In the following,  we will denote by $\mathbb{P}_{x,t}$ the probability measure associated with the process satisfying $R(t)=x$. Note that we have $\mathbb{P}=\mathbb{P}_{-\infty,0}$.

	\subsection{Attraction to the region near $\pm \infty$: proof of Lemma \ref{lem:attractioninfty}.} \label{subsec:prooflemattractioninfty}
	\phantom{blabla.}

	\medskip
	On the probability $\mathbb{P}_{x,t}$, we define $\sigma$ to be the first hitting time of $\pm r_0$ of the diffusion $R$ after time $t$, and $\tau$ the first hitting time of $\pm r_0/2$ after $\sigma$. Let us introduce the first explosion time $\tau_\infty$ of $R$ after time $\sigma$. We also define the stopping times $\sigma_+$ (resp. $\sigma_-$) as the first hitting time of $R$ of $r_0$ (resp. $-r_0$) after time $t$.

	\medskip

	We aim to bound $ \mathbb{P}\big(\sigma\leq T,\; \tau< \infty \big)$  from above. Using the monotonicity properties of the diffusion $R$, it is sufficient to bound
	\begin{align}\label{boundpositif}
		\sup_{r \geq r_0} \sup_{t \in [T/2,T]} \mathbb{P}_{r,t}\big(\tau< \infty, \tau_\infty > \tau\big)
	\end{align}
	and
	\begin{align}\label{boundneg}
		\mathbb{P}_{-r_0,T/2}\big(\tau< \infty \big)\,.
	\end{align}

	To justify this fact, we divide the event $\{\sigma\leq T,\; \tau< \infty\}$ according to the position of the diffusion $R$ at time $T/2$.  
	
	\begin{itemize}
		\item \emph{Case $R(T/2) \leq -r_0$:} Here, $\sigma=T/2$. By monotonicity, we obtain
		\begin{align*}
			\mathbb{P}\big(R(T/2) \leq -r_0,\; \sigma\leq T,\ \tau< \infty \big)
			&\leq \mathbb{P}_{-r_0,T/2}\big(\tau< \infty \big) \,.
		\end{align*}
		\item \emph{Case $R(T/2) \geq r_0$:} Similarly, $\sigma=T/2$ and
		\begin{align*}
			\mathbb{P}\big(R(T/2) \geq r_0,\; \sigma\leq T,\ \tau< \infty \big) &\leq \sup_{r \geq r_0} 	\mathbb{P}_{r,T/2}\big(\tau< \infty \big) \\
			&\leq \sup_{r \geq r_0} 	\mathbb{P}_{r,T/2}\big(\tau< \infty, \tau_\infty < \tau\big) + 	\mathbb{P}_{r,T/2}\big(\tau< \infty, \tau_\infty > \tau\big)\,\\
			&\leq \mathbb{P}_{-r_0,T/2}\big(\tau< \infty \big) +  \sup_{r \geq r_0} 	\mathbb{P}_{r,T/2}\big(\tau< \infty, \tau_\infty > \tau\big)\,.
		\end{align*}
		\item \emph{Case $-r_0 < R(T/2) <r_0$: }
		If $R$ hits $-r_0$ before $r_0$, we have
		\begin{align*}
			\mathbb{P}\big(|R(T/2)|< r_0,\;\sigma_- = \sigma\leq T,\, \tau< \infty\big) 
			&\leq  \P_{-r_0,T/2}(\tau< \infty)\,.
		\end{align*}
		Conversely, if $R$ hits $r_0$ before $-r_0$, we have
		\begin{align*}
			\mathbb{P}\big(|R(T/2)| < r_0,\;\sigma_+ = \sigma\leq T,\, \tau<& \infty \big) \leq \sup_{t \in [T/2,T]} \P_{r_0,t}(\tau< \infty) \\
			\leq &\sup_{t \in [T/2,T]}\mathbb{P}_{r_0,t}\Big(\tau< \infty, \tau_\infty < \tau\Big) + 	\mathbb{P}_{r_0,t}\big(\tau< \infty, \tau_\infty > \tau\big)\,.
		\end{align*}
	\end{itemize}

	Let us bound \eqref{boundpositif} first.
	For $r \geq r_0$ and $t \in [T/2,T]$, on the event $\{\tau< \infty, \tau_\infty > \tau\}$, the diffusion remains in $[r_0/2,+\infty)$ where the drift is bounded below by $1/2 -e^{-r_0/2}$. Hence, this probability can be bounded by the probability that a Brownian motion with constant drift $h_1=1/2-e^{-r_0/2}$ starts at $r_0$ and hits $r_0/2$ in a finite time which is equal to $\exp(- h_1 r_0)$. Since $r\mapsto re^{-r/2}$ remains bounded, we get the upper bound $Ce^{-r_0 /2}$.\\
	

	Let us turn to \eqref{boundneg}. We have to prove that when $R$ starts from $-r_0$ at time $T/2$, it reaches $-r_0/2$ with an exponentially small probability. This would be the case if the diffusion $R$ was a Brownian motion with constant drift $-1/2$. The difficulty here is that when $\beta$ is small (smaller than $2$), $R$ typically reaches the bottom of the well $-\beta t/4$ in a finite time where it cannot be compared anymore to a Brownian motion with constant drift.

	To circumvent this issue, introduce the sequence of positions $r_k := (1+ k (\beta/16)) r_0$ and times $T_k = T/2 + k r_0$ for $k \geq 0$. Define the stopping times $\tau_k := \inf\{s \geq T_k,\; R(t) = -2 r_k\}$. Note that with our choices, we have $2 r_k \leq \beta T_k/4$ and $r_{k+1}\leq 2 r_k$. We have the following upper bound:
	\begin{align*}
		\P_{T_k,-r_k}(\tau < \infty) &\leq  \P_{T_k,-r_k}(\tau < \tau_{k}) + \P_{T_k,-r_k}(T_{k+1}<\tau_k < \tau < \infty) + \P_{T_k,-r_k}( \tau_k < T_{k+1} \wedge \tau), \\
		&\leq \P_{T_k,-r_k}(\tau < \tau_{k}) +\P_{T_{k+1}, -r_{k+1}}( \tau < \infty) + \P_{T_k,-r_k}( \tau_k < T_{k+1}\wedge \tau) \,.
	\end{align*}
	Note that we use the strong Markov property for the second line and the monotonicity of the diffusion to restart at $-r_{k+1}$ instead of $-2r_k$.
	
	On the event $\{T_k\leq t\leq \tau_k< \tau\}$, over the probability $\mathbb{P}_{T-k,-r_k}$, $R$ remains in $(-2r_k,-r_0/2)$: its drift is bounded below by $-1/2$ and above by $h_2\defeq-1/2+e^{-2r_0}$. Comparing with Brownian motions with such drifts, we can easily bound the two terms
	\begin{align*}
		\P_{T_k,-r_k}(\tau < \tau_{k}) &\leq C\exp(- r_k/2), \\
		\P_{T_k,-r_k}( \tau_k < T_{k+1} ,\; \tau_k < \tau) & \leq \exp(-r_k/2)\,.
	\end{align*}
	
	Indeed, the probability in the first inequality is smaller than the probability that a Brownian motion with constant drift $h_2$ starts from $0$ and attains $r_k-r_0/2$ in finite time, which is smaller than $e^{-h_2 r_k}$. We conclude since $r_ke^{-r_k/2}$ remains bounded by $C>0$. For the second inequality, remark that $\tau_k-T_k$ is smaller than the arrival time in $-r_k$ of a Brownian motion with constant drift $-1/2$ started at $0$. We then obtain the upper bound by integrating its probability density function up to times $T_{k+1}-T_k$, given by formula 
	$$ t\mapsto\frac{r_k}{\sqrt{2\pi t^3}}e^{-(r_k-t/2)^2/2t}.$$
	
	It suffices now to show that $\P_{T_{k}, -r_{k}}( \tau < \infty) \to_{k \to \infty} 0$ to conclude. This follows from a comparison with a Brownian with drift $-\beta /8$.
	\qed
	
	\subsection{Proof of Lemma \ref{lem:repulsionaway0}}\label{subsec:lemrepulsion}
	
	We examine in this paragraph the probability that $|R(t)|$ stays below $r_0 = \mu T$, $\mu=(1\wedge\beta)/64$, during the time interval $[T/2,T]$.

	The potential in which $R$ evolves has an unstable equilibrium around $0$, therefore it will not spend too much time there. When $R$ reaches high positive (resp. negative) values, it is very close to a Brownian motion with drift $1/2$ (resp. $-1/2$): it will attain even higher norms using those drifts with high probability. 
	
	To control the time spent near $0$, we introduce the following stopping times. Let $U_0 := T/2$,
	\begin{align*}
		V_1&\defeq\inf\{s\geq U_0,\; |R(s)|\geq 1\}
	\end{align*}
	and, for all integer $n\geq 1$, 
	\begin{align*}
		U_{n}\defeq\inf\{s\geq V_{n},\; |R(s)| = 1/2\}\,,\quad 
		V_{n+1}\defeq\inf\{s\geq U_n,\;|R(s)|= 1\}\,.
	\end{align*}
	Define the time set
	\begin{align*}
		I_0 :=  \bigcup_{n=0}^{+\infty} \Big[U_n,V_{n+1} \Big) \cap [T/2,T]
	\end{align*}
	and its complement over $[T/2,T]$
	\begin{equation}\label{escape intervals}
		I_1:=\bigcup_{n=1}^{+\infty} \Big[V_n,U_n \Big) \cap [T/2,T]\,.
	\end{equation}
	Note that $|R|$ is smaller than $1$ over $I_0$ and greater than $1/2$ over $I_1$.
	
	\medskip
	
	We first control the whole time spent by $|R|$ below $1$ in $[T/2,T]$ on the studied event. This is the content of the next lemma.
	\begin{Lemma}[Time spent near $0$]\label{lem:timespentnear0}
		For all $T \geq 1$, we have
		\begin{align*}
			\P\Big(\int_{T/2}^{T} 1_{|R(t)| \leq 1} dt \geq T/4,\; \sup_{t\in[T/2,T]} |R(t)| \leq \mu T\Big) \leq C\exp(-(1\wedge\beta)T/128)\,.
		\end{align*}
	\end{Lemma}
	An immediate corollary of Lemma \ref{lem:timespentnear0} is that 
	\begin{align*}
		\P\Big(|I_0| \geq T/4, \; \sup_{t\in[T/2,T]} |R(t)| \leq \mu T\Big)\leq C\exp(-(1\wedge\beta)T/128)\,.
	\end{align*}

	\begin{proof}[Proof of Lemma \ref{lem:timespentnear0}]
		We use the Girsanov-Cameron theorem with
		\begin{align}\label{eq:girsanov}
			\mathrm{d}\tilde R(t)=\frac{\lambda}{2} f(t)\cosh \tilde R(t) \mathrm{d}t - \frac{1}{2}\tanh \tilde R(t)\mathrm{d}t+\mathrm{d}B(t)\,.
		\end{align}
		One can suppose that $|R(T/2)| \leq \mu T$ otherwise the probability is equal to $0$. Then, both $R$ and $\tilde{R}$ are bounded on $[T/2,\sigma\wedge T]$, which ensures that Novikov's condition is satisfied.
		We obtain for all functional $\varphi$ over the path in the time-interval $[T/2,\sigma\wedge T]$
		\begin{align*}
			\E[\varphi(R)] = \E\Big[\varphi(\tilde R) \exp\Big(	\int_{T/2}^{\sigma\wedge T} \tanh(\tilde R(t)) d\tilde R(t)
			- \frac{1}{2} \int_{T/2}^{\sigma\wedge T} \gl f(t) \cosh(\tilde R(t)) \tanh(\tilde R(t)) dt \Big)\Big]\,.
		\end{align*}
		It\^o formula gives:
		\begin{align*}
			\ln \cosh(\tilde R(\sigma \wedge T)) - 	\ln \cosh(\tilde R(T/2)) = \int_{T/2}^{\sigma\wedge T} \tanh \tilde R(t) d\tilde R(t) + \frac{1}{2} \int_{T/2}^{\sigma\wedge T} (1- \tanh^2 \tilde R(t)) dt \,.
		\end{align*}
		Therefore we get
		\begin{align*}
			\E[\varphi(R)] &= 
			\E\Big[\varphi(\tilde R) \exp\Big(	\ln \Big(\frac{\cosh(\tilde R(\sigma\wedge T))}{\cosh(\tilde R(T/2))}\Big) - \frac{1}{2} \int_{T/2}^{\sigma\wedge T} (1- \tanh^2 \tilde R(t)) dt \\
			&\qquad \qquad \qquad   \qquad - \frac{1}{2} \int_{T/2}^{\sigma\wedge T} \gl f(t) \cosh(\tilde R(t)) \tanh(\tilde R(t)) dt\Big)\Big] \,.
		\end{align*}
		Define
		\begin{align*}
			A := \Big\{\forall t \in [T/2,\sigma\wedge T], \; |R(t)| \leq \mu T,
			\; \int_{T/2}^{\sigma\wedge T} 1_{|R(t)| \leq 1} dt \geq T/4\Big\}\,.
		\end{align*}
		On $A$, we have $T\leq \sigma$ almost surely. The choice $\varphi = 1_{A}$ gives
		\begin{align*}
			\P(A) \leq \exp\Big(-\frac{1}{8} (1-\tanh(1)^2) T + \mu T + \gl \frac{\beta}{4} T \exp((\mu - \beta/8)T\Big)\,.
		\end{align*}
		The term which depends on $\lambda$ is uniformly bounded in $\beta,T$. Since $\mu=(1\wedge\beta)/64$ is strictly smaller than $\frac{1}{16}( (1-\tanh(1)^2) \wedge \beta)$, we obtain the desired exponential decay.
	\end{proof}

	We can therefore suppose that $|I_1| \geq T/4$. Then we notice that in the time set $I_1$, the process $|R|$ stays above $1/2$. Therefore, as long as it stays below $\mu T$, it is stochastically bounded from below by a Brownian motion with \emph{constant drift} equal to $h_3 := \tanh(1/2)/2$. The probability that it stays below $\mu T$ is therefore bounded from above by
	\begin{align*}
		\P\Big(\sup_{t\in[0,T/4]} (B(t) + h_3 t) \leq \mu T \Big)\,.
	\end{align*}
	One can use the rough bound
	\begin{align*}
		\P\Big(\sup_{t\in[0,T/4]} (B(t) + h_3 t) \leq \mu T \Big) &\leq \P\Big(B(1) \leq (2\mu - h_3/2) \sqrt{T} \Big)\,,
	\end{align*}
	which decreases at exponential rate since $\mu =(1\wedge\beta)/64< h_3/4$.
	
	\qed
	\vspace{0,3cm}

	\section{Approximating diffusions with a discrete scheme}\label{discrete section}

In this section, we would like to find a random variable measurable with respect to the discrete steps of the driving Brownian motion which is close to the value of the diffusion $\alpha^{(i)}$ at time $T$. To this end, we follow the classical Euler-Maruyama method, used to simulate stochastic diffusions.

We will state and prove a more general lemma. Let us fix a real bounded continuous function $f$ and a complex Lipschitz function $g$ and let $W$ be a standard complex Brownian motion (the choices $W=W^{(i)}$ will lead to the approximation of each $\alpha^{(i)}$). With a slight abuse of notation, let us define by $\alpha$ the solution of the following general SDE and starting at $\alpha(0) = 0$:
\begin{align*}
	\text{d}\alpha(t) = f(t)\text{d}t + \Re\Big[g\big(\alpha(t)\big)\;\mathrm{d}W(t)\Big]\,.
\end{align*}
Let us fix $n \geq 1$ and let $\delta := T/n$ be the length of the time-step. The \textit{piecewise constant approximation} $\alpha_{\text{p.c.}}$ is the solution of the following recursive equation: for all $j \geq 0$,
\begin{align*}\label{piecewisecontinuous}
	\alpha_{\text{p.c.}}\big((j+1)\delta\big)=\alpha_{\text{p.c.}}(j\delta)+\int_{j \delta}^{(j+1)\delta}f(t)\text{d}t +\Re\Big[g\big( \alpha_{\text{p.c.}}(j \delta)\big)\Big( W\big((j+1)\delta\big)-W\big(j \delta\big)\Big)\Big]\,,
\end{align*}
with initial condition $\alpha_{\text{p.c.}}(0)=0$ and for all $s \in \big[j \delta,(j+1) \delta\big)$, $\alpha_{\text{p.s.}}(s) = \alpha_{\text{p.c.}}(j \delta)$.

\begin{Prop}\label{discrete} 
	For all $T>0$, we have the following inequality
	$$\mathbb{E}\left[\big|\alpha_{\mathrm{p.c.}}(T)-\alpha(T)\big|^2 \right] \leq 8 c_g^2\Big(\|g\|_\infty^2+\delta\|f\|_\infty^2\Big)T \delta\, e^{4 c_g^2 T}\,,$$
	where $c_g$ denotes the Lipschitz constant of the function $g$.
\end{Prop}
\begin{Rem}\label{Remarque}
	Using Doob's inequality, we can obtain an $L^2$ bound for the whole trajectory until time $T$, namely
	$$\mathbb{E}\left[\sup_{0\leq t \leq T}\big|\alpha_{\mathrm{p.c.}}(t)-\alpha(t)\big|^2\right]\leq 32 c_g^2\Big(\|g\|_\infty^2+\delta\|f\|_\infty^2\Big)T \delta e^{4 c_g^2 T}.$$
\end{Rem}

Lemma \ref{lem:discrete} is a direct consequence of the proposition, with the choices $f(t) = \frac{\beta}{4} \gl  e^{-(\beta/4) t}$, $g(x) = e^{-ix}-1$ and $W=W^{(i)}$.

\begin{proof}[Proof of Proposition \ref{discrete}] 
	Let us introduce the continuous process $\alpha_\mathcal{C}$ by
	\begin{equation}\label{lin}
		\alpha_{\mathcal{C}}(s)=\alpha_{\text{p.c.}}(s)+ \int_{j \delta}^s f(t)\mathrm{d}t + \Re (g(\alpha_{\text{p.c.}}(s))\left(W(s)-W(j \delta)\right)),\quad s\in[j\delta,(j+1)\delta)\,.
	\end{equation}
	Note that we have $\alpha_\mathcal{C}(T)=\alpha_{\text{p.c.}}(T)$.
	Moreover,
	\begin{equation}\label{eds}
		\alpha_{\mathcal{C}}(s)=\int_0^sf(t)\mathrm{d}t+\Re\Big[\int_0^sg(\alpha_{\text{p.c.}}(t))\mathrm{d}W(t)\Big]\,.
	\end{equation}
	
	\noindent Therefore,
	\begin{align*}
		\mathbb{E}\left[ \big| \alpha_{\mathcal{C}}(T)-\alpha(T)\big|^2\right] &
		\leq \E \bigg[\Big|\Re\int_0^Tg(\alpha_{\text{p.c.}}(s))-g(\alpha(s))\ \mathrm{d}W(s)\Big|^2\bigg] \\
		&\leq 2\E\bigg[\int_0^T  \Big|g(\alpha_{\text{p.c.}}(s))-g(\alpha(s))\Big|^2 ds\bigg]\\
		& \leq 2\,c_g^2\,\mathbb{E}\left[ \int_0^T \big| \alpha_{\text{p.c.}}(s)-\alpha(s)\big|^2\mathrm{d}s \right] \\
		&\leq 4\,c_g^2\,K_T+4\,c_g^2\,\int_0^T\mathbb{E}\left[\big|\alpha_{\mathcal{C}}(s)-\alpha(s)\big|^2\right]\mathrm{d}s,
	\end{align*}
	where 
	$$K_T:=\mathbb{E}\left[ \int_0^T \big| \alpha_{\mathcal{C}}(s)-\alpha_{\text{p.c.}}(s)\big|^2\mathrm{d}s \right].$$
	Using Gronwall's lemma, one gets
	\begin{equation}\label{lipschitz}\mathbb{E}\left[\big|\alpha_{\mathcal{C}}(T)-\alpha(T)\big|^2 \right]\leq 4\,c_g^2\,K_T\times e^{4c_g^2 T}.\end{equation}
	Let us bound $K_T$. For every $j\delta\leq s < (j+1)\delta$, we have
	\begin{align*}
		\big| \alpha_{\mathcal{C}}(s)-\alpha_{\text{p.c.}}(s)\big| &= \bigg| \int_{j \delta}^sf(t)\mathrm{d}t\ +\ \Re\big( g(\alpha_{\text{p.c.}}(s))( W(s)-W(j\delta))\big)\bigg|\\
		&\leq \|f\|_\infty \delta + \|g\|_\infty |W(s) - W(j \delta)|\,.
	\end{align*}
	Therefore, we have
	\begin{align*}
		K_T 
		&\leq 2\,\big(\| f \|_\infty^2  \delta+2\|g\|_\infty^2\big)T \delta\,.
	\end{align*}
	Together with \eqref{lipschitz}, this concludes the proof.
\end{proof}
	
	\medskip
	
	\section{Asymptotic independence between two driving Brownian motions}\label{fidi section}
	
	In this section, we focus on the case $k=2$ and the proof of Lemma \ref{lem:totalvariation} which quantifies how the dependence between the two driving Brownian motions $\Wn$ and $\Wd$ decreases as the distance $r$ tends to infinity. Recall that the diffusion $\alpha_r$ defined in \eqref{eqalphar} encodes their correlation via the expression
	$$\Wd(t)= \int_0^t e^{-i\alpha_r(s)}\mathrm{d}\Wn(s)\,.$$
	For large values of $r$ and times smaller than $T_r$, the process $\exp (-i\alpha_r)$ oscillates very fast. 
	As a result, the two Brownian motions become asymptotically independent. 
	Note that this behavior changes for times close to $T_r$, where the two Brownian motions start to show correlation.
	We therefore restrict the time interval to $[0,T]$ where $T \leq T_r$.

	Let us introduce a partition 
	\begin{align*}
		t_j=j \frac{T}{n}, \quad 0\leq j\leq  n \,, 
	\end{align*}
	of the time interval $[0,T]$ and denote the step size $\delta := T/n$.

	One can compare the finite dimensional vector of the time increments of the two driving Brownian motions  $\Wn$ and $\Wd$ at those discrete times with those of two independent Brownian motions $W^{(1)}\otimes W^{(2)}$. We set for $j = 1,\cdots,n$,
	\begin{align*}
		X^n_{2j -1} :=  \Wn(t_{j}) - \Wn(t_{j-1}),\quad X^n_{2j} :=  \Wd(t_{j}) - \Wd(t_{j-1})\,.
	\end{align*}
	The covariance matrix of the (circular) complex Gaussian vector $X^n$ denoted by $\Sigma_X$ is a $2 \times 2$-block matrix of the form
	\begin{align*}
		\Sigma_X =\begin{pmatrix} K_1 & & \\
			& \ddots & \\
			& & K_n
		\end{pmatrix},\qquad K_j \defeq \begin{pmatrix} 2 \delta & \kappa_j \\
			\overline{\kappa_j}& 2 \delta
		\end{pmatrix},
	\end{align*}
	where for all $j=1,\cdots,n$,
	\begin{align*}
		\kappa_j := \E\Big[\;\big(\Wn(t_{j}) - \Wn(t_{j-1})\big) \overline{\big(\Wd(t_{j}) - \Wd(t_{j-1})\big)}\;\Big]\,.
	\end{align*}

	We first bound from above the correlations $\kappa_j$ in the vector $X^n$ thanks to the strong oscillations of $\alpha_r$.
	
	\begin{Lemma}[Covariance of the increments of $\Wn$ and $\Wd$]\label{lem:covarianceBrowniens}
		There exists an absolute constant $C$ such that for all $\beta >0$ and $t \geq s \geq 0$, we have
		\begin{align*}
			&\bigg|\,\E\Big[\;\big(\Wn(t) - \Wn(s)\big) \overline{\big(\Wd(t) - \Wd(s)\big)}\;\Big]\bigg| \leq C\left(1+\frac{1}{\beta}\right) \,\frac{e^{\beta t/4}}{ r}\,.	
		\end{align*}
	\end{Lemma}

	\begin{proof} 
		Let us fix $s \leq t$. The correlation between the two complex increments is given by
		\begin{align*}
			\E\big[\;\big(\Wn(t)-\Wn(s)\big) \overline{\big(\Wd(t) - \Wd(s)\big)}\;\big]  &= \E\big[\big\langle \Wn(t)-\Wn(s), \overline{\Wd(t) - \Wd(s)}\big\rangle\big] \\
			&= 2\,\E\Big[ \int_s^t e^{i \alpha_r(u)} \mathrm{d}u \Big] \,.	
		\end{align*}
		Let us denote the martingale part of $\alpha_r$ by $\tilde \alpha_r(t)=\int_0^t 2\sin(\alpha_r(s)/2) \mathrm{d}B(s)$. Recall that
		\begin{align*}
			\alpha_r(t) = rF(t) + \tilde \alpha_r(t)\,,
		\end{align*}
		where $F(t) = 1-e^{- (\beta/4) t}$.
		
	\noindent	By It\^o formula, we have:
		\begin{align*}
			\int_s^t e^{i \alpha_r(u)} \mathrm{d}u &= \Big[ e^{i \tilde \alpha_r(u)} \frac{e^{i rF(u)}}{ i r f(u)}\Big]_s^t - \int_s^t \frac{e^{ i r F(u)}}{ i r f(u)} \mathrm{d}(e^{ i \tilde \alpha_r(u)}) +  \int_0^t e^{ i \tilde \alpha_r(u) +  i r F(u)} \frac{f'(u)}{ i r (f(u))^2} \mathrm{d}u\,,\\
			&=\Big(\frac{e^{i \alpha_r(t)}}{ i r f(t)}- \frac{e^{i \alpha_r(s)}}{ i r f(s)}\Big)  + \int_s^t \frac{2 e^{i \alpha_r(u)} }{ i r f(u)} \sin(\alpha_r(u)/2)^2 \mathrm{d}u +  \int_s^t e^{ i \alpha_r(u)} \frac{f'(u)}{i r (f(u))^2} \mathrm{d}u\\
			&\quad  - \int_s^t \frac{2 e^{i \alpha_r(u)}}{r f(u)}  \sin(\alpha_r(u)/2) \mathrm{d}B(u)\,.
		\end{align*}
		The first three terms gives the desired upper bound. When taking the expectation, the last term equals zero.
		
	\end{proof}
	
	Consider a complex Gaussian vector $Y^n$ of covariance matrix $\Sigma_{Y} = 2\delta I_{2n}$.
	One can bound the total variation distance between $X^n$ and $Y^n$ using Hellinger distance which has a simple expression for two complex Gaussian vectors. 
	More precisely, the squared Hellinger distance is equal to 
	\begin{align*}
		H^2(X^n,Y^n) =  \frac{1}{2} \int \Big(\sqrt{f_{X^n}(z)} - \sqrt{f_{Y^n}(z)}\big)^2 dz
	\end{align*}
	where $f_{X^n}$ (resp. $f_{Y^n}$) denotes the density of $X^n$ (resp. $Y^n$) with respect to the complex Lebesgue measure on $\mathcal{C}^n$.	Using Cauchy-Schwarz inequality, we obtain
	\begin{equation}\label{boundTVHellinger}
		d_{TV}(X^n,Y^n)  \leq \sqrt{2} H(X^n,Y^n)\,.
	\end{equation}
	For a complex Gaussian vector of covariance matrix $\Sigma$, the density writes
	\begin{align*}
		\frac{1}{(2\pi)^{2n} \det \Sigma} \exp\big(- \frac{1}{2}z^\ast \Sigma^{-1} z\big)\,.
	\end{align*}
	Therefore, a computation gives
	\begin{align}\label{expressionHellinger}
		H(X^n,Y^n) = \sqrt{1- \frac{\det \Sigma_X^{1/2} \det \Sigma_Y^{1/2}}{\det ((\Sigma_X + \Sigma_Y)/2)}}
	\end{align}
	Computing the determinants,	we obtain
	\begin{align*}
		H(X^n,Y^n) = \sqrt{1- \frac{\sqrt{\prod_{i=1}^{n} (1- |\kappa_i|^2/(4\delta^2))}}{\prod_{i=1}^{n} (1- |\kappa_i|^2/(16\delta^2)}}
	\end{align*}
	If $\sup |\kappa_i|/\delta$ is small enough, we obtain the following rough bound for some absolute constant $C$
	\begin{align*}
		H(X^n,Y^n) \leq C \frac{\sqrt{\sum_{i=1}^{n}|\kappa_i|^2}}{\delta}\,.
	\end{align*}
	Using Lemma \ref{lem:covarianceBrowniens}, we obtain the result. \qed

	\vspace{0,3cm}

		\section{Generalization to any number of intervals}\label{section:kpoints}

		Let us prove Lemma \ref{lem:kpoints}. Recall the discretization $(t_j)_{j=0}^n=(j\delta)_{j=0}^n$, with $\delta=T/n$, and the two clusters $J_1=\{1,\cdots,k_0\}$ and $J_2=\{k_0+1,\cdots,k\}$. As we will run similar procedures on each cluster, we use the index $q=1,2$ to recall which cluster $J_q$ is involved. For all $1\leq j\leq n$, $\ell \in \N$ and any function $\boldsymbol{X}:[0,T]\to\mathbb{C}^{\ell}$, we define $\Delta_j\boldsymbol{X}=\boldsymbol{X}(t_{j})-\boldsymbol{X}(t_{j-1})$ and $\Delta\boldsymbol{X} = (\Delta_j\boldsymbol{X}, \;j \in \{1,\cdots,n\})$. We also denote by $X^{(i)}$ the $i$-th coordinate of $\boldsymbol{X}$, and if $\ell = k$, we let $\boldsymbol{X}_q=(X^{(i)})_{i\in J_q}$. 
		
		We organize the proof as follow. At first, we give sufficient conditions on a family of increments $(\Delta \boldsymbol{Z}_1, \Delta \boldsymbol{Z}_2)$ to ensure that the discrete approximations $\gamma_{\text{p.c.}}^{(1)},\cdots,\gamma_{\text{p.c.}}^{(k)}$ constructed from them satisfy
		$$\sup_{1\leq i\leq k}\mathbb{P}\Big(\big|\alpha^{(i)}(T) - \gamma_{\text{p.c.}}^{(i)}(T)\big|>\pi/2\Big)\leq C\frac{e^{4T}}{ n}.$$
		The conditions are essentially that the increments must be independent and close to those of $(\Delta \boldsymbol{W}_1, \Delta\boldsymbol{W}_2)$. Then, we construct such a family which additionally satisfies the total variation bound
		$$\mathrm{d}_{TV}\bigg(\Big(\Delta_j\boldsymbol{Z}_1,\Delta_j\boldsymbol{Z}_2\big)\Big)_{j=1}^n,\Big(\Delta_j\boldsymbol{Z}_1\otimes\Delta_j\boldsymbol{Z}_2\Big)_{j=1}^n\bigg)\leq C\,k^{3}\, \, \frac{n^{5/2} \; e^{\frac{\beta}{4}T}}{r}.$$
		As we shall explain in more details below, the time discretization of $(\boldsymbol{W}_1,\boldsymbol{W}_2)$ does not necessarily satisfy this inequality: the constant $C$ depends on the eigenvalues of the covariance matrices. When $k$ grows, the small eigenvalues can be arbitrarily small with respect to $k$ or even equal to $0$ which makes the constant $C$ arbitrarily large (the case when one of them is equal to $0$ even gives an infinite constant). Fortunately, tiny eigenvalues do not affect much the dynamics of the discrete approximations. We will therefore construct increments $(\Delta \boldsymbol{Z}_1, \Delta \boldsymbol{Z}_2)$ with better correlation structure by projecting $(\Delta\boldsymbol{W}_1,\Delta\boldsymbol{W}_2)$ on the eigenspaces associated to large enough eigenvalues.
		
		\subsection*{Constructing discrete approximations}
		In the previous section, we constructed the discrete approximations $\alpha^{(i)}_{\mathrm{p.c.}}$, which are measurable with respect to the Brownian increments $\Delta \boldsymbol{W}=(\Delta_j W^{(i)})_{1\leq i\leq k, 1\leq j \leq n}$. These are defined by the recursion
		\begin{align*}
			\alpha^{(i)}_{\mathrm{p.c.}}(0)=0,\quad \Delta_{j}\alpha^{(i)}_{\mathrm{p.c.}}=\int_{t_{j-1}}^{t_j}f(t)\mathrm{d}t+\Re\Big(g\big(\alpha^{(i)}_{\mathrm{p.c.}}(t_{j-1})\big)\Delta_j W^{(i)}\Big),
		\end{align*}
		where $g(x)=\exp(-i\,x)-1$. The function $g$ is Lipschitz continuous with constant $c_g\leq 1$ and bounded by $\|g\|_\infty \leq 2$.
		
		To improve the correlation structure, we introduce a modified sequence of approximations driven by a sequence of random vectors $\Delta \boldsymbol{Z} = (\Delta_j Z^{(i)})_{1\leq i\leq k, 1\leq j \leq n}$ instead of $\Delta \boldsymbol{W}$. We denote these new processes by $\gamma_{\mathrm{p.c.}}^{(i)}$. They are defined recursively by
		\begin{align}\label{eq:gamma approx}
			\gamma_{\mathrm{p.c.}}^{(i)}(0)=0, \quad \Delta_{j}\gamma_{\mathrm{p.c.}}^{(i)}=\int_{t_{j-1}}^{t_j}f(t)\mathrm{d}t+\Re\Big(g\big(\gamma_{\mathrm{p.c.}}^{(i)}(t_{j-1})\big) \Delta_j Z^{(i)}\Big).
		\end{align}
		
		We assume that $(\Delta_j \boldsymbol{W},\Delta_j \boldsymbol{Z})_{1\leq j\leq n}$ is a family of independent random variables, and that for all $i\in\{1,\cdots,k\},\ j\in\{1,\cdots,n\}$, the variables $\Delta_j Z^{(i)}$ are centered and close to $\Delta_j W^{(i)}$:
		\begin{equation}\label{eq:L2cond}
			\mathbb{E}\Big[\Delta_j Z^{(i)}\Big]=0, \qquad \mathbb{E}\left[\big|\Delta_j W^{(i)}-\Delta_j Z^{(i)}\big|^2\right]\leq \frac{1}{n^2}.
		\end{equation}
		
		The following lemma ensures that the new approximations $\gamma_{\mathrm{p.c.}}^{(i)}$ remain close to the reference approximations $\alpha^{(i)}_{\mathrm{p.c.}}$, and consequently to the continuous limit.
		
		\begin{Lemma}\label{lem:gammaclose}
			Under the assumptions above, there exists a constant $C$ such that for any $1\leq i\leq k$,
			\begin{equation}
				\mathbb{P}\Big( \big|\alpha^{(i)}_{\mathrm{p.c.}}(T) - \gamma_{\mathrm{p.c.}}^{(i)}(T)\big|>\pi/6\Big)\leq C\frac{e^{4T}}{n}.
			\end{equation}
		\end{Lemma}
		
		\begin{proof}
			Fix $1\leq i\leq k$ and denote the error between the two discrete schemes by $e_j \defeq \alpha^{(i)}_{\mathrm{p.c.}}(t_j) - \gamma_{\mathrm{p.c.}}^{(i)}(t_j)$. Subtracting the recurrences, the drift term cancels and we obtain
			\begin{align*}
				\big|e_j - e_{j-1}\big| &= \Big|\Re\Big[g\big(\alpha^{(i)}_{\mathrm{p.c.}}(t_{j-1}) \big) \Delta_j W^{(i)}\Big] - \Re\Big[g\big(\gamma_{\mathrm{p.c.}}^{(i)}(t_{j-1}) \big) \Delta_j Z^{(i)}\Big]\Big| \\
				&\leq \Big|\,g\big(\alpha^{(i)}_{\mathrm{p.c.}}(t_{j-1}) \big) \Delta_j W^{(i)} - g\big(\gamma_{\mathrm{p.c.}}^{(i)}(t_{j-1}) \big) \Delta_j Z^{(i)} \Big|.
			\end{align*}
			We decompose the term inside the modulus as
			$$ \Big(g\big(\alpha^{(i)}_{\mathrm{p.c.}}(t_{j-1})\big) - g\big(\gamma_{\mathrm{p.c.}}^{(i)}(t_{j-1})\big)\Big)\Delta_j W^{(i)} + g\big(\gamma_{\mathrm{p.c.}}^{(i)}(t_{j-1})\big) \Big(\Delta_j W^{(i)} - \Delta_j Z^{(i)}\Big). $$
			Squaring the inequality, combined with the Lipschitz property of $g$ and the bound $\|g\|_\infty \le 2$, we get
			\begin{align}\label{boundej}
				\big|e_j - e_{j-1}\big|^2 \leq 2 c_g^2 \big|e_{j-1}\big|^2 \big|\Delta_j W^{(i)}\big|^2 + 8 \big|\Delta_j W^{(i)} - \Delta_j Z^{(i)}\big|^2\, .
			\end{align}
			Let $U_j \defeq \mathbb{E}\big[|e_j|^2\big]$. We have
			\begin{align*}
				U_j  &= U_{j-1} + \E\big[	|e_j - e_{j-1}|^2 \big] + 2\, \E\big[ e_{j-1}	(e_j - e_{j-1})\big]\,.
			\end{align*}
			The cross term $2\,\E\big[ e_{j-1}	(e_j - e_{j-1})\big]$ vanishes
			\begin{gather*}
				\E\bigg[ e_{j-1}	\Big(\Re\Big[g\big(\alpha^{(i)}_{\mathrm{p.c.}}(t_{j-1}) \big) \Delta_j W^{(i)}\Big] - \Re\Big[g\big(\gamma_{\mathrm{p.c.}}^{(i)}(t_{j-1}) \big) \Delta_j Z^{(i)}\Big]\Big)\bigg]\\
				= \	\Re\Big(\E\Big[ e_{j-1}\, g\big(\alpha^{(i)}_{\mathrm{p.c.}}(t_{j-1})\big) \Big] \E\big[ \Delta_j W^{(i)}\big] - \E\Big[ e_{j-1}\, g\big(\gamma_{\mathrm{p.c.}}^{(i)}(t_{j-1})\big) \Big] \E\big[\Delta_j Z^{(i)}\big]\Big)=0\,,
			\end{gather*}
			since $\Delta_j W^{(i)}$ and $\Delta_j Z^{(i)}$ are independent of  the random variables $(\Delta_\ell W^{(i)}, \Delta_\ell Z^{(i)})_{\ell \leq  j-1}$ and all increments are centered.
			
			Taking the expectation of the bound \eqref{boundej}, we obtain
			\begin{align*}
				U_j &\leq U_{j-1} + 2 c_g^2\, \mathbb{E}\Big[\big|e_{j-1}\big|^2 \big|\Delta_j W^{(i)}\big|^2\Big] + 8\, \mathbb{E}\Big[\big|\Delta_j W^{(i)} - \Delta_j Z^{(i)}\big|^2\Big] \\
				&\leq U_{j-1} \left(1 + 2\, \mathbb{E}\Big[\big|\Delta_j W^{(i)}\big|^2\Big]\right) + \frac{8}{n^2}\,,
			\end{align*}
			where we used the independence between $e_{j-1}$ and $\Delta_j W^{(i)}$ and the fact that $c_g \le 1$. Since $\mathbb{E}\big[|\Delta_j W^{(i)}|^2\big] = 2T/n$, we have
			$$ U_j \leq U_{j-1} \left(1 + \frac{4T}{n}\right) + \frac{8}{n^2}\,. $$
			Together with $U_0=0$, we obtain recursively
			\begin{align*}
				U_n \leq \frac{8}{n^2} \sum_{k=0}^{n-1} \Big(1+\frac{4T}{n}\Big)^k = \frac{8}{n^2} \frac{(1+4T/n)^n - 1}{4T/n}.
			\end{align*}
			Using the bound $(1+x/n)^n \leq e^x$, this gives
			\begin{align*}
				U_n \leq \frac{2\, e^{4T}}{nT}.
			\end{align*}
			We conclude using Tchebychev's inequality.
		\end{proof}
		
		\subsection*{Existence of appropriate good candidates}
		We now construct the increments $\Delta_j \boldsymbol{Z}_1$ and $\Delta_j \boldsymbol{Z}_2$ to improve the correlation structure while remaining close to those of our driving Brownian motions. Our strategy is to perform a \emph{spectral regularization}: we preserve the principal components of the Brownian motion but replace the small, unstable modes with independent noise of controlled variance. For simplicity of notation, we do not emphasize the time-dependence anymore in the following lines. Recall that it is present in the variable $j \in\{1,\cdots,n\}$.
		
		\medskip
		
		We aim to prove the following lemma.
		\begin{Lemma}\label{Lem:good candidates}
			There exist sequences of increments $(\Delta_j \boldsymbol{Z}_1)_{j=1}^n$ and $(\Delta_j \boldsymbol{Z}_2)_{j=1}^n$ satisfying the assumptions of Lemma \ref{lem:gammaclose} such that
			\begin{equation}\label{distTVZ}
				\mathrm{d}_{TV}\bigg(\Big(\Delta_j\boldsymbol{Z}_1,\Delta_j\boldsymbol{Z}_2\Big)_{j=1}^n,\Big(\Delta_j\boldsymbol{Z}_1\otimes\Delta_j\boldsymbol{Z}_2\Big)_{j=1}^n\bigg)\leq C\,k^{3}\, \frac{n^{5/2} \; e^{\frac{\beta}{4}T}}{r}.
			\end{equation}
		\end{Lemma}
		
		\noindent \textit{Construction via spectral regularization.}
		Fix a threshold $\varepsilon \defeq \frac{1}{2k n^2}$.
		Let $q\in\{1,2\}$ be the cluster index and fix a time step $j$. Let $M_q^j$ be the spatial covariance matrix of the Brownian increments $\Delta_j \boldsymbol{W}_q$. Let $P^j_q$ be the unitary matrix diagonalizing $M_q^j$ with ordered eigenvalues $\lambda_1 \ge \dots \ge \lambda_k \ge 0$ that is $M_q^j = P_q^j \;\mbox{diag}(\lambda_1, \cdots, \lambda_k ) \; (P_q^j)^H$.
		We define the cutoff index $p^j_q$ such that $\lambda_\ell > \varepsilon$ for $\ell \le p^j_q$ and $\lambda_\ell \le \varepsilon$ for $\ell > p^j_q$.
		
		We define the increments $\Delta_j \boldsymbol{Z}_q$ by manipulating the coordinates in the eigenbasis. Let $\boldsymbol{Y}_q \defeq P_q^j (\Delta_j \boldsymbol{W}_q)$ be the vector of eigen-coordinates. We construct the regularized coordinates $\tilde{\boldsymbol{Y}}_q$ as follows:
		\begin{itemize}
			\item For $\ell \leq p^j_q$ (Large modes): we set $\tilde{Y}_q^{(\ell)} = Y_q^{(\ell)}$.
			\item For $\ell > p^j_q$ (Small modes): we set $\tilde{Y}_q^{(\ell)} = \xi^{(\ell)}_q$, where $\xi^{(\ell)}_q \sim \mathcal{N}_{\mathbb{C}}(0, \varepsilon)$ are independent complex Gaussian variables of variance $\eps$, independent of $\Delta \boldsymbol{W}$. We also ask that the families $\big(\tilde{Y}^{(\ell)}_q\big)_{\ell}$ are independent for distinct $j$'s.
		\end{itemize}
		Finally, we rotate back to the original basis: $\Delta_j \boldsymbol{Z}_q \defeq (P_q)^H \tilde{\boldsymbol{Y}}_q$. Note that the covariance matrix of $\Delta_j \boldsymbol{Z}_q$ is given by
		\begin{align*}
			\tilde M_q^j := P_q^j \;\mbox{diag}(\lambda_1, \cdots, \lambda_{p_q^j}, \eps,\cdots,\eps) \; \big(P_q^j\big)^H.
		\end{align*}
		
		\medskip
		\noindent \textit{Verification of the assumptions.}
		By construction, $\Delta_j \boldsymbol{Z}_q$ is a centered Gaussian vector. Since we defined independent $\tilde{\boldsymbol{Y}}_q$ at each step $j$, the joint increments $(\Delta_j \boldsymbol{Z}_q, \Delta_j \boldsymbol{W}_q)$ are independent in time. For the $L^2$-bound, we sum the errors in the eigenbasis. For the small modes $\ell > p^j_q$, the error variance is $\mathbb{E}\big[\big|Y^{(\ell)}_q - \xi^{(\ell)}_q\big|^2\big] = \text{Var}\big(Y_q^{(\ell)}\big) + \text{Var}\big(\xi_q^{(\ell)}\big) = \lambda_{\ell} + \varepsilon \leq 2\varepsilon$. Thus:
		$$ \mathbb{E}\big[\|\Delta_j \boldsymbol{W}_q - \Delta_j \boldsymbol{Z}_q \|^2\big] = \sum_{\ell=p_q+1}^k (\lambda_\ell + \eps) \leq 2 k\eps  = \frac{1}{n^2}. $$
		This implies the coordinate-wise condition \eqref{eq:L2cond} required by Lemma \ref{lem:gammaclose}.
		
		\medskip

		\textit{Structure of the covariance matrices.}
		Before computing \eqref{distTVZ}, we write the form of the covariance matrices $\Sigma_{a}$ and $\Sigma_{b}$ of $(\Delta \boldsymbol{Z}_1, \Delta \boldsymbol{Z}_2)$ versus $(\Delta \boldsymbol{Z}_1 \otimes \Delta \boldsymbol{Z}_2)$. Recall that the diffusions $(\alpha_{x_{i_1}})_{i_1\in J_1},\ (\alpha_{x_{i_2}+r})_{i_2\in J_2}$ defined in \eqref{eqalphar} encode the correlations in between $\Delta \boldsymbol{W}_1$ and $\Delta \boldsymbol{W}_2$ via the expressions
		\begin{equation}\label{eq:code} W^{(i_1)}(t)=\int_0^t \exp
			\Big(i\big(\alpha_{x_{i_2}+r}(s)-\alpha_{x_{i_1}}(s)\big)\Big)\mathrm{d}W^{(i_2)}(s),\qquad  \big(i_1\in J_1,\ i_2\in J_2\big),\end{equation}
		where $(\boldsymbol{W}_1,\boldsymbol{W}_2)=(W^{(i)})_{1\leq i\leq k}$ and $x_{i_1}\leq 0 \leq x_{i_2}$.

		The covariance matrices $\Sigma_{a}$ and $\Sigma_{b}$ are block matrices of the form
		\begin{align*}
			\Sigma_{a} =\begin{pmatrix} K_1 & & \\
				& \ddots & \\
				& & K_n
			\end{pmatrix}, \qquad \Sigma_{b} =\begin{pmatrix} L_1 & & \\
				& \ddots & \\
				& & L_n
			\end{pmatrix}.
		\end{align*}
		For $1\leq j\leq n$, we define the $|J_1| \times |J_2|$ matrices $\boldsymbol{\kappa}_j$ by
		\begin{align*}
			\boldsymbol{\kappa}_j := \E\Big[ \Big(\Delta_j\boldsymbol{Z}_1\,\Big)\cdot \Big(\,\Delta_j\boldsymbol{Z}_2\Big)^H\Big]\,.
		\end{align*}
		Then, the blocks (of size $k \times k$)  $K_j$ and $L_j$ are given by
		\begin{align*}
			K_j\defeq\begin{pmatrix} 
				\tilde	 M^j_1& \boldsymbol{\kappa}_j
				\\ \boldsymbol{\kappa}_j^H &\tilde  M^j_2
			\end{pmatrix}\,, \qquad L_j\defeq\begin{pmatrix} 
				\tilde M^j_1 &\boldsymbol{0}^{H}
				\\\boldsymbol{0} & \tilde M^j_2
			\end{pmatrix}\,.
		\end{align*}
		We show in the next lemma that the two matrices $K_j$ and $L_j$ are close in infinite norm.
		\begin{Lemma}\label{lem:kpoints proof}
			For all $\beta>0$, there exists a constant $C=C(\beta)$ such that for all $k\geq 2$, $j\geq 1$, we have
			\begin{align*}
				\big\|\boldsymbol{\kappa}_j\big\|_\infty=\Big\|\,\E\Big[\big(\,\Delta_j\boldsymbol{Z}_1\,\big) \cdot \big(\,\Delta_j\boldsymbol{Z}_2\big)^H\Big]\Big\|_{\infty} \leq C \,k\,\frac{e^{\beta T/4}}{r}\,.	
			\end{align*}
		\end{Lemma}
		
		\begin{proof} Fix $j\in \{1,\cdots,n\}$. We omit the dependence in $j$ of certain quantities in this proof to alleviate the notations. 
			
			Recall the definition of the increments in the eigenbasis: $\Delta_j \boldsymbol{W}_q = (P_{q})^H \boldsymbol{Y}_{q}$ and $\Delta_j \boldsymbol{Z}_q = (P_{q})^H \tilde{\boldsymbol{Y}}_q$. Let us decompose $\tilde{\boldsymbol{Y}}_{q} = \hat{\boldsymbol{Y}}_{q} + \boldsymbol{N}_q$ where $\hat Y^{(\ell)}_q =  Y^{(\ell)}_q 1_{\ell \leq p_q} $ and $N_q^{(\ell)} = \xi^{(\ell)}_q 1_{\ell > p_q}$. 
			Substituting this into the definition of the cross-covariance $\boldsymbol{\kappa}_j$
			\begin{align*}
				\boldsymbol{\kappa}_j &= \mathbb{E}\Big[ \Delta_j \boldsymbol{Z}_1\cdot (\Delta_j \boldsymbol{Z}_2)^H \Big] \\
				&= (P_1)^H \, \mathbb{E}\Big[ \big(\hat{\boldsymbol{Y}}_1 + \boldsymbol{N}_1\big)\cdot \big(\hat{\boldsymbol{Y}}_2 + \boldsymbol{N}_2\big) \Big] \, P_2.
			\end{align*}
			As the noise terms $\boldsymbol{N}_1$ and $\boldsymbol{N}_2$ are constructed from independent centered variables, which are also independent of $\boldsymbol{W}_q$ (and thus of $\hat{\boldsymbol{Y}}_q$), all cross-terms vanish and we obtain
			$$ \boldsymbol{\kappa}_j = (P_1)^H \, \mathbb{E}\Big[\big( \hat{\boldsymbol{Y}}_1\big)\cdot \big(\hat{\boldsymbol{Y}}_2\big)^H \Big] \, P_2. $$
			Recall that $\hat{\boldsymbol{Y}}_q = D_q \boldsymbol{Y}_q$, where $D_q$ is the diagonal matrix with entries $(D_q)_{\ell\ell} = \mathbf{1}_{\ell \leq p_q}$. Thus
			$$ \boldsymbol{\kappa}_j = (P_1^H D_1 P_1) \; \mathbb{E}\Big[ \big(\Delta_j \boldsymbol{W}_1\big)\cdot  \big(\Delta_j \boldsymbol{W}_2 \big)^H \Big] \, (P_2^H D_2 P_2). $$
			
			Let $C_{1,2} := \mathbb{E} \Big[ \big(\Delta_j \boldsymbol{W}_1\big)\cdot  \big(\Delta_j \boldsymbol{W}_2 \big)^H \Big]$. Let $(e_\ell)_{\ell\in\{1,\cdots,k\}}$ be the canonical base in $\mathbb{C}^k$. We bound the coefficients of $\boldsymbol{\kappa}_j$ element-wise
			\begin{align*}
				\big|(\boldsymbol{\kappa}_j)_{m,\ell}\big| &= \big|\transp{e_m} \;(P_1^H D_1 P_1) \; C_{1,2}\; (P_2^H D_2 P_2)\; e_\ell\,\big| \\
				&\leq \big\| (P_1^H D_1 P_1) \; e_m\big\|\; \big\|(P_2^H D_2 P_2)\; e_\ell\,\big\|\; \big\|C_{1,2}\big\|_\infty \,.
			\end{align*}
			Since $P_q^H D_q P_q$ are orthogonal projections, we have the upper bound
			$$ \big|(\boldsymbol{\kappa}_j)_{m,\ell}\big| \leq \sqrt{|J_1|} \sqrt{|J_2|}\; \big\|C_{1,2}\big\|_\infty \leq k\; \big\|C_{1,2}\big\|_\infty \,. $$
			Using Lemma \ref{lem:covarianceBrowniens}, which ensures $\|C_{1,2}\|_\infty \leq C e^{\beta T/4}/r$, we obtain the result.
		\end{proof}

		\textit{Computations of Hellinger's distance.} Now, we can bound the total variation distance of Lemma \ref{Lem:good candidates} using Hellinger distance. Using the bound \eqref{boundTVHellinger} and the expression of Hellinger distance given in \eqref{expressionHellinger}, as well as the block-diagonal form of the covariance matrices $\Sigma_{a}$ and $\Sigma_{b}$, we have:
		\begin{align*}
			\frac{1}{\sqrt{2}}\mathrm{d}_{TV}\bigg(\Big(\Delta_j\boldsymbol{Z}_1,\Delta_j\boldsymbol{Z}_2\Big)_{j=1}^n,\Big(\Delta_j\boldsymbol{Z}_1\otimes\Delta_j\boldsymbol{Z}_2\Big)_{j=1}^n\bigg)\leq \sqrt{1- \prod_{j=1}^n\frac{\det (K_j)^{1/2} \det (L_j)^{1/2}}{\det \big((K_j+L_j)/2)\big)}}.
		\end{align*}
		If we can prove that for some small $\eta \leq 1/4$, for all $j\in\{1,\cdots,n\}$, we have
		\begin{align}\label{ratiodet}
			\bigg|\,\frac{\det K_j}{\det L_j}\,-\,1\,\bigg| \leq \eta, \qquad \bigg|\,\frac{\det \big((K_j+L_j)/2\big)}{\det L_j}\,-1\,\bigg|\leq \eta\,,
		\end{align}
		then we control the Hellinger distance by
		\begin{align}\label{boundHell}
			\sqrt{1- \prod_{j=1}^n\;\frac{\big(\det K_j/\det L_j\big)^{1/2} }{\det\big((K_j+L_j)/2)\big)/\det L_j}}\leq \sqrt{1- \frac{\big(1-\eta\big)^{n/2} }{(1+\eta)^n}} \leq \sqrt{C\, n \, \eta}\,.
		\end{align}

		We therefore need to bound the ratios as in \eqref{ratiodet}. We focus on the ratio $\det K_j / \det L_j$. Let us drop again the subscript $j$ to ease the reading. Using the Schur complement formula for block determinants, we obtain
		\begin{align*}
			\frac{\det K}{\det L} &= \frac{\det\big(\tilde{M}_1\big) \det\big(\tilde{M}_2 - \boldsymbol{\kappa}^H \big(\tilde{M}_1\big)^{-1} \boldsymbol{\kappa}\big)}{\det\big(\tilde{M}_1\big) \det\big(\tilde{M}_2\big)} \\
			&= \det\Big( I - \big(\tilde{M}_2\big)^{-1} \boldsymbol{\kappa}^H \big(\tilde{M}_1\big)^{-1} \boldsymbol{\kappa} \Big).
		\end{align*}
		
		Define $S_q := (\tilde{M}_q)^{-1/2}$ for $q=1,2$ (we use here that $\tilde M_q$ is a Hermitian matrix). Using Sylvester's determinant formula $\det (I - AB) = \det (I - BA)$, it writes
		\begin{align*}
			\frac{\det K}{\det L} = 	\det\Big( I - (S_1 \boldsymbol{\kappa} S_2)^H (S_1 \boldsymbol{\kappa} S_2)\Big).
		\end{align*}
		The matrix $(S_1 \boldsymbol{\kappa} S_2)^H (S_1 \boldsymbol{\kappa} S_2)$ is Hermitian positive definite. If its trace is smaller than $1$, using the inequality
		\begin{align}\label{ineqeasymatrix}
			1-\mbox{Tr}(P)\leq \det(I - P) \leq 1 
		\end{align}
		valid for all positive definite Hermitian matrices with eigenvalues smaller than $1$, we have
		\begin{align*}
			0	\leq 1-	\frac{\det K}{\det L}  \leq \mbox{Tr}\Big(\big(\tilde{M}_2\big)^{-1} \boldsymbol{\kappa}^H \big(\tilde{M}_1\big)^{-1} \boldsymbol{\kappa} \Big)\,.
		\end{align*}
		We now bound the trace of $R := \big(\tilde{M}_2\big)^{-1} \boldsymbol{\kappa}^H \big(\tilde{M}_1\big)^{-1} \boldsymbol{\kappa}$, in particular showing it is smaller than $1$ for large $r$. Recall that by construction, the eigenvalues of $\tilde{M}_1$ and $\tilde{M}_2$ are bounded from below by $\varepsilon$. 
		
		Using the trace inequality $\mathrm{Tr}(A B) \leq \|A\|_{\mathrm{op}} \mathrm{Tr}(B)$ valid for a Hermitian positive semi-definite matrix $B$, we get:
		\begin{align*}
			\mbox{Tr}(R) &\leq \frac{1}{\varepsilon} \; \mbox{Tr}\Big( \boldsymbol{\kappa}^H \big(\tilde{M}_1\big)^{-1} \boldsymbol{\kappa}\Big) \\
			&\leq  \frac{1}{\varepsilon} \; \mbox{Tr}\Big( \big(\tilde{M}_1\big)^{-1} \boldsymbol{\kappa} \boldsymbol{\kappa}^H\Big)\\
			& \leq \frac{1}{\varepsilon^2} \; \mbox{Tr}\big( \boldsymbol{\kappa} \boldsymbol{\kappa}^H\big)\,.
		\end{align*}
		This implies:
		\begin{align*}
			\text{Tr}(R) \leq \frac{k^2}{\varepsilon^2} \| \boldsymbol{\kappa}\|_{\infty}^2\,.
		\end{align*}
		Using Lemma \ref{lem:kpoints proof} for the bound on $\|\boldsymbol{\kappa}\|_\infty$, we find
		\begin{align*}
			\text{Tr}(R) \leq \frac{k^2}{\varepsilon^2} \left( C k \frac{e^{\beta T/4}}{r} \right)^2 = C \frac{k^4 e^{\beta T/2}}{\varepsilon^2 r^2}.
		\end{align*}
		Substituting $\varepsilon = (2kn^2)^{-1}$, we infer
		$$ \text{Tr}(R) \leq C \cdot k^6 n^4 \frac{e^{\beta T/2}}{r^2}. $$
		Assume this bound is smaller than $1$ so that indeed the inequality \eqref{ineqeasymatrix} holds. The same bound holds for the ratio involving $(K_j+L_j)/2$ (replacing $\boldsymbol{\kappa}_j$ with $\boldsymbol{\kappa}_j/2$).

		We have therefore bounded all the ratios by $\eta :=  C \, k^6 n^4 \frac{e^{\beta T/2}}{r^2}$ under the assumption $\text{Tr}(R) \leq 1$. Thanks to \eqref{boundHell}, it yields the result of Lemma \ref{Lem:good candidates} under this assumption. But note that when $\text{Tr}(R) \geq 1$, our bound is large but the total variation is always bounded by $2$ so we can safely replace it by our large bound, which concludes the proof.
		\qed

		\bibliographystyle{plain}
		\bibliography{Mabibliotheque}

	\end{document}